\documentclass[12pt,a4paper]{amsart}
\usepackage{mathrsfs}
\usepackage{amsfonts}
\usepackage[active]{srcltx}
\usepackage[notref,notcite]{showkeys}
\usepackage{enumerate}
\usepackage{color}
\usepackage{amsmath,amssymb,xspace,amsthm}

\theoremstyle{definition}
\newtheorem{thm}{Theorem}[section]
\newtheorem{lem}[thm]{Lemma}
\newtheorem{defn}[thm]{Definition}
\newtheorem{cor}[thm]{Corollary}
\newtheorem{prop}[thm]{Proposition}
\newtheorem{rem}[thm]{Remark}

\numberwithin{equation}{section}

\def\spn{\mathrm{span}}

\newcommand{\lcm}{\operatorname{lcm}}

\def\Ind{\mathrm{Ind}}

\def\hom{\mathrm{Hom}}

\newcommand{\C}{\mathbb C}

\newcommand{\N}{\mathbb{N}}
\newcommand{\Z}{\mathbb{Z}}

\def\fZ{\mathfrak{Z}}

\renewcommand{\a}{\alpha}

\renewcommand{\d}{\delta}
\def\D{\Delta}
\def\tD{\widetilde{\D}}
\renewcommand{\l}{\lambda}

\renewcommand{\phi}{\varphi}
\newcommand{\eps}{\epsilon}

\def\ot{\otimes}

\def\Ind{\mathrm{Ind}}

\def\Ann{\mathrm{Ann}}
\def\Rad{\mathrm{Rad}}

\def\LL{\mathcal{L}}
\def\HH{\mathcal{H}}

\def\II{\mathcal{I}}
\def\GG{\mathcal{G}}

\def\UU{\mathcal{U}}

\def\fb{\mathfrak{b}}

\def\g{\mathfrak{g}}
\def\tg{\widetilde{\g}}
\def\hg{\widehat{\g}}
\def\tM{\widetilde{M}}
\def\hM{\widehat{M}}
\def\hW{\widehat{W}}
\def\tW{\widetilde{W}}
\def\tphi{\widetilde{\phi}}
\def\hphi{\widehat{\phi}}

\def\h{\mathfrak{h}}
\def\th{\widetilde{\h}}
\def\hh{\widehat{\h}}

\def\n{\mathfrak{n}}
\def\tn{\widetilde{\n}}

\def\hb{\widehat{\fb}}

\def\sl{\mathfrak{sl}}

\def\calE{\mathcal{E}}

\newcommand{\pl}{\pmb{\l}}
\newcommand{\pa}{\bold{a}}

\newcommand{\pmu}{\pmb{\mu}}

\def\brac#1#2{\langle \; #1 \,,\, #2 \,\rangle}
\def\gen#1{\langle \, #1 \,\rangle}

\title[untwisted  affine Kac-Moody algebras]{\bf Irreducible representations of untwisted  affine Kac-Moody algebras}
\author{Xiangqian Guo, Kaiming Zhao}

\date{}

\begin{document}
\maketitle

\begin{abstract}
In this paper we  construct a class of  new irreducible modules
over  untwisted  affine Kac-Moody algebras $\widetilde{\mathfrak{g}}$,  generalizing
and including both highest weight modules and Whittaker modules.
These modules allow us to obtain a complete classification of
irreducible $\widetilde{\mathfrak{g}}$-modules on which the action of each root vector in
$\widetilde{\mathfrak{n}}_+$ is locally finite, where $\widetilde{\mathfrak{n}}_+$ is the locally nilpotent
subalgebra (or positive part) of $\widetilde{\mathfrak{g}}$. The necessary and sufficient
conditions for two such irreducible $\widetilde{\mathfrak{g}}$-modules to be isomorphic
are also determined. In the second part of the paper,  we use the ``shifting technique" to obtain
a necessary and sufficient condition for the tensor product of
irreducible integrable loop $\widetilde{\mathfrak{g}}$-modules and irreducible integrable highest
weight $\widetilde{\mathfrak{g}}$-modules to be simple. This tensor product problem was
originally studied by Chari and Pressley 28
years ago.

 \vskip.5cm


\end{abstract}


\section{Introduction}

Affine Lie algebras are the most extensively studied and most useful
ones among  infinite-dimensional Kac-Moody Lie algebras. Their
representation theory is as  rich as but quite different to that of
finite-dimensional simple Lie algebras. One difference is that,
affine Lie algebras have irreducible weight modules containing both
finite and infinite-dimensional nonzero weight spaces, that cannot
occur in the finite-dimensional simple Lie algebra case. In the
present paper we will construct irreducible modules over untwisted
affine Kac-Moody Lie algebras that do not have counterpart for
finite-dimensional simple Lie algebras.

Let $\g$ be a finite dimensional simple Lie algebra. We use $\tg$ to
denote the corresponding untwisted affine Kac-Moddy Lie algebra and let
$\hg=[\tg,\tg]$.

The integrable highest weight modules were the first class of
representations over affine Kac-Moody Lie algebras that have been
extensively studied, see [K] for detailed discussion of results and
further bibliography. In [Ch] Chari classified all irreducible
integrable weight modules with finite-dimensional weight spaces over
the untwisted affine Lie algebras. V. Chari and A. Pressley, [CP2],
then extended this classification to all affine Lie algebras. The
results of [Ch] and [CP2] state that every irreducible integrable
weight module with finite-dimensional weight spaces is either a
highest weight module or a loop module; in particular, the study of
loop modules was initiated. Some more general results on weight
irreducible modules over untwisted affine Lie algebras $\hg$ were
also obtained in \cite{Li}.

In the 1990's, V. Futorny began a comprehensive program studying
weight modules over arbitrary affine Lie algebras by taking
non-standard partitions of the root system; that is, partitions
which are not equivalent under the Weyl group to the standard
partition into positive and negative roots (see [DFG]). For affine
Lie algebras, there are always only finitely many equivalence
classes of such non-standard partitions (see [F4]). Corresponding to
each partition is a Borel subalgebra, and one can form
representations induced from one-dimensional modules for these Borel
subalgebras. These modules, often referred to as Verma-type modules,
were first studied by Jakobsen and Kac [JK], and then by Futorny
[F3, F4]. Results on the structure of Verma-type modules can also be
found in [Co, F1, FS]. For a nice exposition of Futorny's program,
see [F5].

Very recently, a complete classification for all simple weight
modules with finite-dimensional weight spaces over affine Lie
algebras were obtained in \cite{FT, DG}.
  Naturally, the next important task
is to study irreducible weight modules with infinite-dimensional
weight spaces and irreducible non-weight modules. The first examples were given by Chari and Pressley in \cite{CP3} by taking the tensor product of some irreducible integrable highest weight modules and integrable loop modules over affine Lie algebras.
 Besides these
irreducible weight modules, a  class of irreducible
weight modules over affine Lie algebras with infinite-dimensional
weight spaces were constructed  in [BBFK]. Another class of such irreducible
weight modules over the affine Lie algebra $A_1^{(1)}$  were constructed  in \cite{FGM}.
A complete classification
for all irreducible (weight and non-weight) modules over affine Lie
algebras with locally nilpotent action of the nilpotent radical was
obtained in [MZ]. In the paper [Chr], some irreducible non-weight
modules, called imaginary Whittaker modules, were constructed. The
structure of the usual Whittaker modules over the affine Kac-Moody
algebra $A_1^{(1)}$ was systematically studied in [ALZ], while this
study for other affine Kac-Moody algebras is still open.

The present
paper is to classify irreducible modules over untwisted affine Lie
algebras $\tg$ on which the action of each weight vector in $\tn_+$
is locally finite (not necessarily locally nilpotent), where $\tn_+$
is the locally nilpotent subalgebra (or positive part) of $\tg$, and
construct some new irreducible weight $\tg$-modules.

The paper is organized as follows. In Sect.2, we recall some basic
notations and results for later use. In Sect.3, using loop modules
$E(\pl,\pa)$ and some $\hg$-modules $L$ with a special property
(including highest weight modules and Whittaker modules) we
construct a class of new modules over untwisted affine Lie algebras
$\tg$ and study their properties (see Theorem 3.3). These modules
are in general non-weight modules. More precisely, let  $E$ be a
nontrivial irreducible evaluation $\g\ot\C[t,t^{-1}]$-module (such
modules were classified in \cite{CP2} and a more general case was
obtained in \cite{L}). Then we prove that the $\tg$-module
$\Ind_{\hg}^{\tg}(E\ot L)$ is irreducible. In Sect.4, we study
properties of irreducible $\tg$-modules $V$ on which the action of each root
vector in $\tn_+$ is locally finite (not necessarily locally
nilpotent). In Sect.5,
using our results in Sect.3 we classify all irreducible $\tg$-modules
on which the action of each root vector in $\tn_+$ is locally finite. They are precisely irreducible highest weight modules, irreducible Whittaker modules, the irreducible
modules $(E(\pl,\pa)\ot V(\gamma))[d]$ and $(E(\pl,\pa)\ot
W(\eta))[d]$ with $E(\pl,\pa)$ being finite dimensional (see Theorem
5.3). The last two classes  are special examples of the irreducible modules we
constructed in Sect.3. As a byproduct we see that the action of $\tn_+$ on these irreducible $\tg$-modules are also locally finite.
In Sect.6, we determine the necessary and
sufficient conditions for two such irreducible $\tg$-modules to be
isomorphic (see Theorem 6.3).

In 1987, Chari and Pressley  \cite{CP3} first gave  some sufficient
conditions for the tensor product of irreducible loop $\tg$-modules
and irreducible integrable  highest weight $\tg$-modules to be
simple. Then some other sufficient conditions were obtained in
\cite{CP4, Ad1, Ad2, Ad3}. This problem has been unsolved for 28
years.  In Sect.7,    we use the ``shifting technique" to obtain a
necessary and sufficient condition for  this tensor product to be
simple for untwisted affine Kac-Moody algebras (Theorem 7.9). From this  we  recover
   Admovic's result in \cite{Ad2}.

Throughout this paper, we denote by $\Z$, $\N$, $\Z_+$ and $\C$ the
sets of integers, positive integers, nonnegative integers and
complex numbers respectively. All vector spaces and Lie algebras are
over $\C$. For a Lie algebra $\GG$, we denote its universal
enveloping algebra by $U(\GG)$ and the center of $\GG$ by
$\fZ(\GG)$.

\section{Notation}  \label{pre}

In this section we will recall standard notions, basically  from \cite{K} and
\cite{BM}.

Let $\g$ be a finite dimensional simple Lie algebra of type $X_l$
over $\C$. For this $\g$ we   fix a root system $\Delta$ and a base   $\Pi$ of $\Delta$.
Denote by $\Delta_+$ and $\Delta_-$ the sets of positive
roots and negative roots respectively. Then $\g$ admits a triangular
decomposition $\g = \n_- \oplus \h \oplus \n_+$, where $\h$ is the
Cartan subalgebra and $\n_\pm$ are corresponding maximal nilpotent
subalgebras with respect to $\Pi$. Let $\sigma$ be a diagram
automorphism of $\g$. The order $r$ of $\sigma$ is $1$, $2$, or $3$
($r$ can be $2$ for $X_l = A_l, D_l, E_6$ and $r$ can be $3$ for
$X_l = D_4$.) Denote the eigenspace decomposition of $\g$ with
respect to $\sigma$ by $\g = \bigoplus \limits_{i=0}^{r-1} \g_i$, where we
put
\begin{equation*}
\g_i := \{\, x \in \g \mid \sigma(x) = \omega^i x \,\} \mbox{ for
} i =0,1, \ldots, r-1
\end{equation*}
with $\omega=\exp(\frac{2\pi\sqrt{-1}}{r}) \in \C$. For any vector
subspace $V$ of $\g$, denote $V \cap \g_i$ by $V_i$. For example,
$\n_{i,\pm} = \n_\pm \cap \g_i$ and $\h_i = \h \cap \g_i$. Then
$\g_0$ becomes a simple Lie algebra with $\h_0$ as its Cartan
subalgebra, see \cite{K} for more details.

Let us recall the definition of the affine Lie algebra $\tg$
associated to the pair $(\g,\sigma)$. Put $R_i := t^{i}\C[t^r,t^{-r}] \subset \C[t,t^{-1}]$. Not that $R_i=R_{r+i}$ for any $i\in\Z$.
Define the subalgebra $\LL_r(\g)$ of the loop algebra $\g \otimes \C[t,t^{-1}]$ by
\begin{equation*}
\LL_r(\g) := \bigoplus_{i=0}^{r-1} \g_i \otimes R_i,
\end{equation*}
which is called a (twisted) loop Lie algebra.
Let $d$ be the degree derivation $t\frac{d}{dt}$ acting on $\LL_r(\g)$.
Then $\LL_r(\g) \oplus \C d$ possesses a natural semidirect Lie algebra structure.
As a vector space, we define $\tg$ by
\begin{equation*}
\tg:= \LL_r(\g) \oplus \C K \oplus \C d.
\end{equation*}
Its Lie algebra structure is defined by
\begin{equation*}
[ x \otimes t^m, y \otimes t^n ] = [x,y] \otimes t^{m+n} + (x|y) m
\delta_{m+n,0} K, \quad [ d, x \otimes t^m ] = m x \otimes t^m
\end{equation*}
where $K$ is central in $\tg$, and $(x|y)$ is the Killing form of
$\g$. When $r=1$, $\tg$ is called an untwisted affine Lie algebra
and in the other cases ($r=2,3$), $\tg$ is called a twisted affine
Lie algebra. For convenience, we denote the subalgebra $$\hg:=
\LL_r(\g) \oplus \C K.$$

For any $0 \neq \alpha \in \h_0^*:=\hom_{\C}(\h_0,\C)$, we put
\begin{equation*}
\g_{i,\alpha} := \{x \in \g_{i} \mid [h,x] = \alpha(h)x, \ \forall\
h \in \h_0 \}
\end{equation*}
\begin{equation*}
\Delta_{i} := \{ 0 \neq \alpha \in \h_0^* \mid \g_{i,\alpha} \ne 0 \}
\end{equation*}
i.e., the set of nonzero weights of $\h_0$ on $\g_i$. Then, for
$\alpha \in \Delta_{i}$  we can write $\g_{i,\alpha}=\C
e_{i,\alpha}$ since $\dim \g_{i,\alpha}=1$ (see \cite{K}). Now we
can describe the root system $\tD $ and the root space decomposition of
$\tg$ with respect to $\th = \h_0 \oplus \C K\oplus \C d$ as
follows:
\begin{equation*}
\tD = \{\a + k \d : k \in r\Z +i, \a \in \D_i, i =0, 1, \ldots,
r-1\} \bigcup \{j \delta\,|\, j \in \Z \setminus \{0\}\};
\end{equation*}
\begin{equation*}
\tg= \th \oplus \Big(\bigoplus_{\gamma \in \tD} \tg_\gamma\Big)
\quad  \text{ with } \quad \tg_{\a + k\d} = \g_{\bar{k},\a}
\otimes t^k, \; \tg_{j\d} = \h_{\bar{j}} \otimes t^j
\end{equation*}
where $\d$ is the standard imaginary root and $\bar{k} =i$ if $k \in
r \Z+i$ for some $i=0,1,\ldots,r-1$. The set of positive roots of
$\tg$ is $$\tD_+ =\D_{0,+} \bigcup \{n\d\ |\ n \in \N\} \bigcup \{\a
+ n\d\ |\ \a \in \D_{\bar{n}}, n \in \N\}$$ where $\D_{0,+}$ is the
positive roots of $\g_0$ corresponding to $\n_{0,+}$. Now we have
\begin{equation*}
\tn_+ = \bigoplus_{\gamma \in \tD_+} \tg_\gamma =
\bigoplus_{i=1}^{r-1} (\g_i \otimes t^i\C[t^r]) \oplus (\g_0 \otimes
t^r\C[t^r]) \oplus (\n_{0,+} \otimes \C),
\end{equation*}
and the standard triangular decomposition
$$\tg = \tn_+ \oplus \th \oplus
\tn_-$$ where the meaning for $\tn_-$ is obvious.

For the subalgebra $\hg$ we can define similar notation only by
replacing $\th$ with $\hat{\h}=\h+\C K$. Using the triangular
decompositions, we can define highest weight modules and Whittaker
modules over $\tg$ and $\hg$. Let $\GG=\tg, \HH=\th$ (or sometimes $\GG=\hg,
\HH=\hat{\h}$).

\begin{defn} \label{highest weight}
Let $\gamma: \HH \rightarrow \C$ be a linear map and $V$ be a
$\GG$-module.
\begin{enumerate}
\item A nonzero vector $v \in V$ is called a {\it highest weight vector with highest weight $\gamma$}
if $xv=\gamma(x)v$ for all $x \in \HH$, and $\tn_+ v=0$.
\item The module $V$ is called a {\it  highest weight module with highest weight  $\gamma$ over $\GG$}
if it is generated by a   highest  weight  vector with highest weight  $\gamma$.
\end{enumerate}
\end{defn}
It is well known that irreducible highest weight modules are
determined uniquely by its highest weight. So we denote by
$V(\gamma)$ the irreducible highest weight module over $\GG$ (that
will be specified later) with highest weight $\gamma$.

\begin{defn} \label{Whittaker}
Let $\eta: \tn_+ \rightarrow \C$ be a nonzero Lie algebra
homomorphism and $W$ be a $\GG$-module.
\begin{enumerate}
\item A nonzero vector $v \in W$ is called a {\it Whittaker vector of type $\eta$}
if $xv=\eta(x)v$ for all $x \in \tn_+$.
\item $W$ is called a {\it type $\eta$ Whittaker module for $\GG$}
if it is generated by a type $\eta$ Whittaker vector.
\end{enumerate}
\end{defn}
We say that  the Lie algebra homomorphism $\eta: \tn_+ \rightarrow \C$ for   affine Lie algebras    $\tg$
is {\it nondegenerate} if $\eta(e_i)\ne0$ for each Chevalley
generator $e_i$ of $\tn_+$.   We denote by $W(\eta)$ any one of  the
irreducible Whittaker module of type $\eta$ in this case (it may not
be unique).

The structure of irreducible Whittaker modules $W(\eta)$ over the affine Lie algebra
$A_1^{(1)}$ was thoroughly studied in \cite{ALZ}, while it is unclear for other affine Lie algebras.

\section{Construction of irreducible modules}

In this section we construct a class of new modules over untwisted
affine Lie algebras $\tg$ and study their properties. So we assume
that $r=1$ in the rest of the paper.

Let $\widetilde{\C^m}$ be the set of all vectors
$\pa=(a_1,\cdots,a_m)\in\C^m$ with all $a_1, \cdots, a_m$ being
nonzero and pairwise distinct. Take $\pa\in \widetilde{\C^m}$ and
set $f_{\pa}(t)=(t-a_1)\cdots (t-a_m)$. It is well-known that the
quotient Lie algebra
\begin{equation}\label{L_a}\LL_{\pa}(\g)= \frac{\g\ot \C[t^{\pm1}]}{\g\ot
f_{\pa}(t)\C[t^{\pm1}]}\end{equation}
 is the
direct sum of $m$ copies of $\g$ and hence is semisimple. Actually
$\LL_{\pa}(\g)=\oplus_{i=1}^m\LL_{\pa}^{(i)}(\g),$ where
\begin{equation}\label{LL_i}\LL_{\pa}^{(i)}(\g)=\g\otimes
\left(\frac{f_{\pa}(t)}{(t-a_i)\prod_{j\ne
i}(a_i-a_j)}\right)\simeq\g.\end{equation}  This is
because
$$\left(\frac{f_{\pa}(t)}{(t-a_i)\prod_{j\ne
i}(a_i-a_j)}\right)^2\equiv \frac{f_{\pa}(t)}{(t-a_i)\prod_{j\ne
i}(a_i-a_j)}\mod (f_{\pa}(t)).$$
We note that $\LL_\pa$ inherit a standard triangular
decomposition and root system from $\g$. In particular,
$\h_\pa=\h\ot (\C[t^{\pm1}]/(f_{\pa}(t)\C[t^{\pm1}]))$ is the Cartan
subalgebra of $\LL_{\pa}(\g)$. Within the quotient Lie algebra
$\LL_{\pa}(\g)$, we have

\begin{lem}\label{iso}
For any $k\in\N$, we have the following canonical isomorphism of
vector spaces
\begin{equation}\label{LL_a(g)}
\frac{\g\ot t^k\C[t]}{(\g\ot t^k\C[t])\cap (\g\ot
f_{\pa}(t)\C[t^{\pm1}])}\cong \frac{\g\ot \C[t^{\pm1}]}{\g\ot
f_{\pa}(t)\C[t^{\pm1}]}.
\end{equation}
\end{lem}

\begin{proof} This is the composition of the following
obvious isomorphisms:

\begin{equation*}\frac{\g\ot t^k\C[t]}{(\g\ot t^k\C[t])\cap (\g\ot
f_{\pa}(t)\C[t^{\pm1}])} = \frac{\g\ot t^k\C[t]}{(\g\ot t^k\C[t])\cap
(\g\ot f_{\pa}(t)\C[t])}$$
$$ \cong \frac{\g\ot t^k\C[t]+\g\ot f_{\pa}(t)\C[t]}{\g\ot f_{\pa}(t)\C[t]}  \cong \frac{\g\ot \C[t]}{\g\ot f_{\pa}(t)\C[t]}
 \cong \frac{\g\ot \C[t^{\pm1}]}{\g\ot
 f_{\pa}(t)\C[t^{\pm1}]}.\qedhere\end{equation*}
\end{proof}

For $\pa=(a_1,\cdots, a_m)\in\widetilde{\C^m}$ and
$\pl=(\l_1,\cdots, \l_m)\in (\h^*)^m$ where $\h^*$ is the dual space
of $\h$, let us first recall from \cite{CP1} the evaluation module
$E(\pl,\pa)$  over $\g\ot \C[t^{\pm1}]$. Let $V(\lambda_i)$ be the
irreducible highest weight module over the finite-dimensional simple
Lie algebra $\g$ with highest weight $\lambda_i$. The action on the
module $E(\pl,\pa)=\bigotimes\limits_{i=1}^mV(\lambda_i) $ is as follows
\begin{equation}\label{action}(x\ot t^k)(\bigotimes\limits _{i=1}^mv_i)=\sum_{i=1}^m a_i^k(v_1\ot...\ot xv_i\ot...\ot v_m),\end{equation}
for any $x\in\g, k\in\Z$. Then $E(\pl,\pa)$ is an irreducible module
over the Lie algebra $\hg=\g\ot\C[t^{\pm1}]\oplus\C K$ with trivial
action of $K$.

Note that $(\g\ot f_{\pa}(t)\C[t^{\pm1}])\oplus\C K\subseteq
\Ann(E(\pl,\pa))$, so we can regard $E(\pl,\pa)$ as a module over
$\LL_{\pa}(\g)\cong\hg/\big((\g\ot f_{\pa}(t)\C[t^{\pm1}])\oplus\C
K\big)$. Thus we can consider the weight vectors and weight spaces
of $E(\pl,\pa)$ with respect to  the Cartan subalgebra $\h_\pa$ 
of $\LL_{\pa}(\g)$. Identify $\h_\pa^*$ with $(\h^*)^m$ via
\eqref{L_a} and \eqref{LL_i}, for example, a weight vector $v\in
E(\pl,\pa)$ of weight $\pmu=(\mu_1,\cdots,\mu_m)\in(\h^*)^m$
satisfies $(h\ot
f(t))v=\sum \limits_{i=1}^m\mu_i(h_i)f(a_i)v$. We will use these notations freely in the rest of the paper. 

\begin{rem}\label{remark} We notice that $(\g\ot f_{\pa}(t))\oplus\C K\neq
\Ann(E(\pl,\pa))$ in general.  Indeed if we take
$f_{\pl,\pa}(t)=\prod \limits_{\l_i\neq 0}(t-a_i)$, then $(\g\ot
f_{\pl,\pa}(t)\C[t^{\pm1}])\oplus\C K=\Ann(E(\pl,\pa))$. (The
equality can be verified by the fact that any simple ideal of
$\LL_{\bar{\pa}}(\g)$, for example $\Ann(E(\pl,\pa))/(\g\ot
f_{\pl,\pa}\C[t^{\pm1}])$, is one of $\LL_i$ in\eqref{LL_i}. The
equality will be used in the proof of Lemma \ref{iso_E ot L}.)
Hence we can regard $E(\pl,\pa)$ as a module over
$\LL_{\bar{\pa}}(\g)\cong\hg/\Ann(E(\pl,\pa))$ and moreover, we have
$E(\pl,\pa)\cong E(\bar{\pl},\bar{\pa})$, where $\bar{\pl}$ and
$\bar{\pa}$ are obtained by removing the entries $\l_i$ and $a_i$
with $\l_i=0$ from $\pl$ and $\pa$ respectively.
\end{rem}

\begin{thm}\label{tensor} Let $L$ be a module over $\hg$ such that for any $v\in L$ we have
$(\g\ot t^k\C[t])v=0$ for sufficiently large $k\in\N$. Let
$\pa\in\widetilde{\C^m}$, $\pl\in (\h^*)^m$. Then any
$\hg$-submodule of $E(\pl,\pa)\ot L$ is of the form $E(\pl,\pa)\ot
N$ for   some $\hg$-submodule $N$ of $L$.
\end{thm}

\begin{proof} Let $W$ be a  nonzero $\hg$-submodule of $ E(\pl,\pa)\ot L$ and
take any nonzero element $w\in W$.   We can write $w=\sum \limits_{i=1}^n
v_i\ot u_i$ for some nonzero $v_i\in E(\pl,\pa)$ and nonzero $u_i\in
L$. Fix one such expression for $w$ with $n$ minimal. We see that
$v_1, v_2, ..., v_n$ are linearly independent. There exists $k\in\N$
such that $(\g\ot t^j) u_i=0$ for all $j\geq k$ and $i=1,\cdots, n$.
Thus for any $x\in \g\ot t^k\C[t]$, we have $xw=\sum \limits_{i=1}^n (xv_i
)\ot u_i\in W.$ From Lemma 3.1 we deduce that
\begin{equation}\label{action_geq k}
\sum_{i=1}^n (xv_i )\ot u_i\in W,\ \forall\ x\in U(\hg).
\end{equation}

It is well known that any endomorphism of an irreducible module over
a countably generated associative $\C$-algebra is a scalar
(Proposition 2.6.5 and Corollary 2.6.6 in [D]). Thus ${\rm
Hom}_{U(\hg)}(E(\pl,\pa),E(\pl,\pa))\cong \C$ since $E(\pl,\pa)$ is
irreducible  over the finite dimensional Lie algebra
$\LL_\pa(g)$. Note that $E(\pl,\pa)$ is a faithful irreducible
module over $U(\hg)/\Ann_{U(\hg)}(E(\pl,\pa) )$. From the Jacobson
Density Theorem (Page 197, [J])  or Proposition 2.6.5 in
[D], we know that $U(\hg)/\Ann_{{\small U(\hg)}}(E(\pl,\pa) )$ is
isomorphic to a dense ring of endomorphisms of the $\C$-vector space
$E(\pl,\pa)$. Then there exists $x_i\in U(\hg)$ such that
$x_iv_j=\delta_{i,j}v_j$ for all $i,j=1,\cdots,n$. Taking $x=x_i$ in
\eqref{action_geq k}, we deduce that $v_i\ot u_i\in W$. Again using
\eqref{action_geq k} for $n=1$, we get $E(\pl,\pa)\ot u_i\subseteq
W$ for $i=1,\cdots,n$. Thus we deduce that if $\sum \limits_{i=1}^n v_i\ot
u_i\in W$ for linearly independent $v_i\in E(\pl,\pa)$ then
$E(\pl,\pa)\ot u_i\subseteq W$ for each $i$.

Let $$N=\{u\in L\ |\ v\ot u\subseteq W \text{ for some nonzero }v\in E(\pl,\pa)\}.$$ From the above
arguments we know that $N$   is a nonzero subspace of $L$ and $E(\pl,\pa)\otimes N\subset W$. Then for any nonzero $v\in E(\pl,\pa)$,
$u\in N$ and $x\in \hg$, we deduce that $x(v\ot u)=(xv)\ot u+v\ot
(xu)\in W$. Since $(xv)\ot u\in W$, we have $v\ot (xu)\in W$. Thus $xu\in N$. In
other words, $N$ is a $\hg$-submodule of $L$. Combining this with the result established in the previous paragraph we see that
$W=E(\pl,\pa)\ot N$. Noticing that the action of $K$ is central, the
lemma follows immediately.
\end{proof}

We point out that the statement in Lemma 3.3 for irreducible
integrable highest weight module $L$ and finite dimensional $E(\pl,\pa)$ is  Theorem 4.2 of
\cite{CP3}.

For any $\hg$-module $M$, we can define an induced $\tg$-module
$$M[d]=\Ind_{\hg}^{\tilde{\g}}M.$$
It is easy to see that $M[d]=\C[d]\ot M$ as vector spaces. We will
simply write $d^nv=d^n\ot v$ for all $n\in\Z_+$ and $v\in M$. Now we
can prove the following 

\begin{thm}\label{E ot L[d]} Let $\pa\in\widetilde{\C^m}$ and $\pl\in (\h^*)^m\setminus\{0\}$.
Let $L$ be a module over $\hg$ such that for any $v\in L$ we have
$(\g\ot t^k\C[t])v=0$ for sufficiently large $k\in\N$. Then any
$\tg$-submodule of $(E(\pl,\pa)\ot L)[d]$ is of the form
$(E(\pl,\pa)\ot N)[d]$ for  some $\hg$-submodule $N$ of $L$.
\end{thm}

\begin{proof} Denote $M=E(\pl,\pa)\ot L$ for short.
Let $M^{(n)}=\sum\limits_{i=0}^n d^i\ot M$ for all $n\in\Z_+$ and
$M^{(n)}=0$ for $n<0$. Similarly as in Theorem \ref{M[d]}, we have
$$(x\ot t^k)(d^n v)=d^n (x\ot t^k) v - [d^n, x\ot t^k] v,$$ for all $x\in
\g, v\in M, k\in\Z, n\in\N,$ which induces a $\hg$-module
isomorphism
$$M^{(n)}/M^{(n-1)}\rightarrow M, \quad d^n v\mapsto v.$$

Let $W$ be  a nonzero submodule of $M[d]$ and take a nonzero $w\in
W$. Suppose that $w\in W\cap M^{(n)}$.  Then we can write $w$ as
$$w=\sum_{j=0}^n d^{j}\sum_{i=1}^{l_{j}}(v_{j,i}\ot u_{j,i}),$$
 where $v_{j,i}\in E(\pl,\pa), u_{j,i}\in L, l_j\in\N,
j=1,\cdots,n$ such that $v_{j,1},\cdots, v_{j,l_j}$ are linearly
independent weight vectors for each $j$. We may assume
that   $v_{j,i}$ has
 weight
$\pmu_{j,i}\in(\h^*)^m$. 
We note that there exists $k'\in\N$ such that $(\g\ot t^{k'}\C[t])
u_{j,i}=0$ for all $i=1,\cdots, l_j$ and $j=1,\cdots,n$. \vskip5pt

\noindent {\bf Claim.} $(E(\pl,\pa)\ot u_{j,i})[d]\subseteq W$ for
all $i,j$. \vskip5pt

We prove this by induction on $n$ and $l_n$. If $n=0$ and $l_0=1$ the result is clear from \eqref{LL_a(g)} and
\eqref{action_geq k}. Now we fix some $n\in\Z_+$ and $l_n\in\N$.
Suppose that the claim holds for smaller $n$ or for the same $n$ and
smaller $l_n$.
\vskip5pt

\noindent {\bf Case 1.} $l_n=1$.\vskip5pt

We may assume $n\geq 1$. Take any $x\in U(\g\ot t^{k'}\C[t])$ such
that $xv_{n,1}=v_{\pl}$, the nonzero highest weight vector in
$E(\pl,\pa)$. Then we have $xu_{j,i}=0$ for all $i=1,\cdots, l_j$
and $j=1,\cdots,n$ and $xw\equiv d^n(v_{\pl}\ot u_{n,1})\mod
M^{(n-1)}$. By replacing $w$ with $xw$ (noticing that $u_{n, 1}$ has
not been changed) we may assume that
$$w=d^n(v_{\pl}\ot u_{n,1})+\sum_{j=1}^{n-1}
d^{j}\sum_{i=1}^{l_{j}}(v_{j,i}\ot u_{j,i})$$ with $\pmu_{n,1}=\pl$.
 Without loss of generality, we may assume that either
$v_{\pl}, v_{n-1,1},\cdots,v_{n-1,l_{n-1}}$ are linearly
independent, or $v_{n-1,1}=v_{\pl}$. For any $h\in\h$ and $k\geq
k'$ we can compute that
\begin{equation*} \begin{split}
(h\ot t^k)w = & (h\ot t^k) d^{n}(v_{\pl}\ot u_{n,1})+(h\ot t^k) \sum_{j=1}^{n-1} d^{j}\sum_{i=1}^{l_{j}}(v_{j,i}\ot u_{j,i})\\
 \equiv & d^n  (h\ot t^k) (v_{\pl}\ot u_{n,1}) - [d^n, h\ot t^k] (v_{\pl}\ot u_{n,1})\\
&\hskip 5pt +d^{n-1}\sum_{i=1}^{l_{n-1}}(h\ot t^k)(v_{n-1,i}\ot u_{n-1,i}) \mod M^{(n-2)}\\
 \equiv &  d^n \pl(h\ot t^k)(v_{\pl}\ot u_{n,1})- nk d^{n-1} \pl(h\ot t^k)(v_{\pl}\ot u_{n,1})\\
& \hskip 5pt  +d^{n-1}\sum_{i=1}^{l_{n-1}}\pmu_{n-1,i}(h\ot
t^k)(v_{n-1,i}\ot u_{n-1,i}) \mod M^{(n-2)},
\end{split}\end{equation*}
where we have used the notation $\pl(h\ot t^k)=\sum \limits_{i=1}^m
a_i^k\l_i(h)$ for $\pl=(\l_1,\cdots,\l_m)\in(\h^*)^m$ and similarly
for $\pmu_{n-1,i}(h\ot t^k)$. Then $W$ contains the following vector
\begin{equation*} \begin{split}
w'=&(h\ot t^k)w  -\pl(h\ot t^k)w \equiv - nk d^{n-1} \pl(h\ot
t^k)(v_{\pl}\ot u_{n,1})\\
&+d^{n-1}\sum_{i=1}^{l_{n-1}}(\pmu_{n-1,i}-\pl)(h\ot
t^k)(v_{n-1,i}\ot u_{n-1,i}) \mod M^{(n-2)}.
\end{split}\end{equation*} for all $k>k'$ and $h\in\h$.
Since $\pl\neq 0$, there exist $h\in\h$ and $k\geq k'$ such that
$\pl(h\ot t^k)=\sum \limits_{j=1}^m a_i^k\l_i(h)\neq 0$.   It is
easy to check that $w'\neq0$ for this $h\ot t^k$ and hence $w'\in
W\cap M^{(n-1)}\setminus\{0\}$.
By induction hypothesis, we deduce that $(E(\pl,\pa)\ot
u_{n,1})[d]\subseteq W$. Consequently, the claim follows in this case.
\vskip5pt

\noindent {\bf Case 2}. $l_n\geq2$.\vskip5pt

Since  $v_{n,1},\cdots, v_{n,l_n}$ are linearly
independent and $E(\pl,\pa)$ is irreducible, using Density Theorem and
  Lemma \ref{iso}, we can choose  an element $x\in U(\g\ot
t^{k'}\C[t])$ such that $xv_{n,i}=\delta_{i, 1}v_{\pl}$. Then $W$ contains
\begin{equation*}
x w\equiv d^n\big(v_{\pl}\ot u_{n,1}\big) \mod
M^{(n-1)}.
\end{equation*}
By induction hypothesis, we deduce that   $(E(\pl,\pa)\ot u_{n,1})[d]\subseteq W$. Using induction hypothesis we obtain the claim.
\vskip5pt

From the above claim we see that $$N=\{u\in L\ |\ (E(\pl,\pa)\ot
u)[d]\subseteq W\}$$ is a nonzero subspace of $L$. It is easy to see that  $N$ is a $\hg$-submodule of $L$. Using the above claim we can easily prove that  $(E(\pl,\pa)\ot N)[d]= W$. This completes the proof of  the theorem. \end{proof}

Now we can have two important consequences from the above theorem.

\begin{cor} \label{M[d]}
Let $L$ be an irreducible module over $\hg$ such that for any $v\in
L$ we have $(\g\ot t^k\C[t])v=0$ for sufficiently large $k\in\N$.
Let $\pa\in\widetilde{\C^m}$, $\pl\in (\h^*)^m\setminus\{0\}$ and
$M=E(\pl,\pa)\ot L$. Then the module $M[d]$ is irreducible over
$\tg$.
\end{cor}

Taking $L$ to be the $1$-dimensional trivial $\hg$-module in Corollary
\ref{M[d]}, we obtain the following

\begin{cor} \label{E[d]} Let $\pa\in\widetilde{\C^m}$ and $\pl\in (\h^*)^m\setminus\{0\}$. Then   $E(\pl,\pa)[d]$ is an irreducible $\tg$-module.
\end{cor}

\section{Properties of locally finite modules}

Recall that $\g=\n_+\oplus \h\oplus \n_-$ is the standard triangle
decomposition of the finite dimensional simple Lie algebra $\g$ with respect to a fixed   Cartan subalgebra $\h$, where
$\n_\pm=\sum \limits_{\a\in\Delta_\pm}\g_\a$. 
We consider the current algebra $\g\otimes \C[t]$ and its subalgebra
$\tn_+=(\g\otimes t\C[t])\oplus \n_+$. We identify $\g\otimes 1$
with $\g$. For any $f(t)\in\C[t]$, we can define an ideal of
$\tn_+$: $$\II(f)=\n_+\otimes \gen{f(t)}+(\h+\n_-)\otimes \gen{t
f(t)},$$ where $\gen{f(t)}$ denotes the ideal generated by $f(t)$ in
$\C[t]$ and similar for $\gen{tf(t)}$.   It is easy to
see $\II(f)=\tn_+\cap (\g\ot f\C[t^{\pm1}])$.  For any $a\neq0$, we
have the following canonical Lie algebra isomorphism
\begin{equation}\label{canonical_iso} \tn_+/\II(t-a) \longrightarrow
\g,\quad x\ot g(t)\mapsto g(a)x,\end{equation} for all  $x\ot
g(t)\in\tn_+\ {\rm with}\ x\in \g, g(t)\in\C[t].$

\begin{thm}\label{local finite}
Let $V$ be a $\tg$-module such that each weight vector in $\tn_+$ acts
locally finitely on $V$. Then $V$ has a finite dimensional nonzero
$\tn_+$-submodule $W$ such that $I(f(t))W=0$ for some nonzero
$f(t)\in\C[t]$.
\end{thm}

\begin{proof} Let $\d$ be the standard imaginary root of $\tg$ with respect to $\th$. Then $\tg_\d
= \h \otimes t$ is a finite dimensional abelian subalgebra of
$\tg$. Since $\tg_{\delta}$ is locally finite on $V$, there is a
nonzero vector $v\in V$ such that $(h\ot t)v\in\C v$ for all $h\ot t
\in \tg_\d$.

Take any $\a\in\Delta$ which is considered as a subset of $\tD$. We
have
$$\a(h)x\ot t^{k+1}v=[h\ot t, x\ot t^k]v=(h\ot t)(x\ot t^k)v-(x\ot t^k)(h\ot t)v,$$ for all $ h\in\h, x\in \tg_{\a},$
which implies $\tg_{\a+(k+1)\delta}\subseteq
\tg_{\delta}\tg_{\a+k\delta}v+\tg_{\a+k\delta}v$ for all $k\in\Z_+$.
Inductively we have $\tg_{k\d+\a} v\subseteq
\sum\limits_{i=0}^{k}\tg_\d^k \tg_\a v$ for all $k \in \Z_+$. Since
$\dim \sum\limits_{k=0}^{\infty}\tg_\d^k \tg_\a v<\infty$, there is
$k_0 \in \N$ such that $\sum\limits_{k=0}^{\infty}\tg_\d^k \tg_\a
v=\sum\limits_{k=0}^{k_0} \tg_\d^k \tg_\a v$ is a finite dimensional
space, i.e., $$\dim \Big(\sum\limits_{k \in \Z_+} \tg_{k\d +
\a}v\Big) =\dim \Big(\sum\limits_{k =0}^{k_0} \tg_{k\d + \a}v\Big)<
\infty.$$ Let $0 \neq x_\a \in \tg_\a$. Then there is a nonzero
$g(t) \in \C[t]$ such that $(x_\a \otimes g(t))v=0$. Note that we
may assume that $g(t)$ is divisible by $t$ when $\a\in\Delta_-$. Set
$$\II_\a = \{g(t) \in \C[t]\ |\ (x_\a \otimes g(t))v=0 {\rm\
and\ }t {\rm\ divides}\ g(t)\ {\rm if}\ \a\in\Delta_-\},$$ which is
nonzero. If $g(t) \in \II_\a$, for any $h \in \h$ we have
\begin{equation*}\begin{split}
0 = (h \otimes t)(x_\a \otimes g(t))v = & \big([h,x_\a] \otimes tg(t)\big)v + (x_\a \otimes g(t))(h \otimes t)v \\
= &\a(h)\big( x_\a \otimes tg(t)\big)v.
\end{split}\end{equation*}
Hence $tg(t) \in \II_\a$, i.e., $\II_\a$ is an ideal of $\C[t]$.
Assume $\II_\a = \gen{f_\a(t)}$ for some nonzero polynomial $f_\a(t)
\in \C[t]$. Setting $\tilde f(t) = \lcm \{f_\a(t) : {\a \in
\Delta}\} \in \C[t]$, the least common multiple of $f_\a(t),
\a\in\Delta$, we have $(\g_\a \otimes \gen{\tilde f(x)})v =0$ for
all $\a \in \Delta$. It follows that
$$(\h \otimes \gen{\tilde f^2(t)})v = \sum\limits_{\a \in \Delta}
\Big[\tg_\a  \otimes \gen{\tilde f(t)}, \tg_{-\a} \otimes \gen{\tilde f(t)}\Big]v=0.$$ Then
there exists nonzero $f\in\C[t]$ such that
 $\II(f)v=0$. Note that $t$ divides $f(t)$,
so we may assume that $\deg(f^2(t))=p+1\geq 1$. Our result follows
from\vskip5pt

\noindent {\bf Claim.} $U(\tn_+)v$ is a finite-dimensional
$\tn_+$-module.\vskip5pt

We know $\II(f)U(\tn_+)v=0$ since $\II(f)$ is an ideal of $\tn_+$.
Take nonzero $e_\a\in\g_\a, f_\a\in\g_{-\a}$ for all $\a\in\Delta_+$
and let $\{h_1,\cdots,h_l\}$ be a basis of $\h$. Denote $\deg(x\ot
t^i)=i$ for all $x\in\g$ and $i\in\Z_+$. Let
$$\{x_i\ |\ i\in\N\}=\{e_\a, e_\a\ot t^{i}, f_\a\ot
t^i, h_j\ot t^i\ |\ \a\in\Delta_+, j=1,\cdots, l, i\in\N\}$$ be such
that $\deg(x_i)\leq \deg(x_{i+1})$ for all $i\in\N$. Take $s$ such
that $\deg (x_{s})=p$ and $\deg(x_{s+1})=p+1$. We know  that
$$W=\sum_{r_i\in\Z_+}\C x_1^{r_1}\cdots x_{s}^{r_s} v$$ is finite-dimensional since
the actions of all $x_i$ are locally finite. For $q\in\Z_+$ let
$$W_q={\text{span}}\{x_1^{r_1}\cdots x_{s}^{r_s} v |
r_1+r_2+...+r_s\le q\}.$$ We know that $W_q=W$ for sufficiently
large $q$. We will show that $\tn_+W_q\subset W_{q+1}.$ Let
\begin{equation}\label{W}
 w=x_1^{r_1}\cdots x_{s}^{r_{s}} v\in W_q, \ \
 r_{i}\in\Z_+.
\end{equation} It is enough to show that $x_jw\in W_{q+1}$ for any $j\in\N$. We will do this by induction on $q$ and then on $j$. For $q=0$,
it is clear that $x_jw\in\C x_jv\subseteq {W}_{q+1}$, since any
$x_jv$ is a linear combination of $x_1v, \cdots, x_sv$ (because of the definition of $f(t)$ and $p$).

Suppose $x_jw\in W_{q+1}$ for any $w$ given in \eqref{W} and any
$j\in\N$ if $q<k$. Now consider one such $w\in W_q$ given in
\eqref{W} with $q=k$.

We write $w=x_{i_1}^{r_1}\cdots x_{i_l}^{r_{l}} v$ where
$i_1<i_2<...<i_l\le s$, $r_1,r_2,...,r_l\in\N$ and
$r_1+r_2+...+r_l=q$. Now we use induction on $j$. If $j\le i_1$, it
is clear that $x_jw\in W_{q+1}$. Now suppose $j>i_1$. We have
\begin{equation}\label{x_j w}x_jw=x_jx_{i_1}x_{i_1}^{r_1-1}\cdots x_{i_l}^{r_{l}} v\\
=[x_j,x_{i_1}]x_{i_1}^{r_1-1}\cdots x_{i_l}^{r_{l}} v
+x_{i_1}x_jx_{i_1}^{r_1-1}\cdots x_{i_l}^{r_{l}} v.\end{equation}
Since $[x_j,x_{i_1}]W_{q-1}\subseteq\text{span}\{x_i\ |\
i\in\N\}W_{q-1}$, by induction hypotheses we know that the first
term in \eqref{x_j w} is $[x_j,x_{i_1}]x_{i_1}^{r_1-1}\cdots
x_{i_l}^{r_{l}} v\in W_{q}$. For the second term, we see that
$x_jx_{i_1}^{r_1-1}\cdots x_{i_l}^{r_{l}} v\in W_{q}$ by the
induction hypothesis on $q$. Then by the induction hypothesis on $j$
we know that the second term in \eqref{x_j w} is
$x_{i_1}x_jx_{i_1}^{r_1-1}\cdots x_{i_l}^{r_{l}} v\in
x_{i_1}W_q\subset W_{q+1}$ since $i_1<j$. Thus $x_jW_q\subset
W_{q+1}$ for all $j$ and all $q$. So $W$ is an $\tn_+$-module, as
desired.
\end{proof}

It is clear that
\begin{equation*}\begin{array}{c}
 [\tn_+,\tn_+]=\sum\limits_{\a\in\Delta_+\setminus\Pi}\g_\a\otimes\C[t]+\sum\limits_{\a\in\Pi}\g_\a\otimes \gen{t}+\sum\limits_{\a\in\Delta_+\setminus\{\theta\}}\g_{-\a}\otimes\gen{t}\\\\
 +\g_{-\theta}\otimes\gen{t^2}+\h\otimes\gen{t},
\end{array}\end{equation*}
where $\Pi$ is the prime root system of $\g$, and
$$\II(t)=\n_+\ot \gen{t}+(\n_-+\h)\ot \gen{t^2}\subseteq[\tn_+,\tn_+].$$

\begin{lem}\label{I(t-a)} Let $a\in\C$, $n\in\N$, and let $S$ be a nontrivial finite dimensional
irreducible $\tn_+$-module with $\II((t-a)^n)S=0$. \begin{itemize}\item[(i).]If $a\neq 0$ then
$\II(t-a)=\Ann(S)$. \item[(ii).]If $a=0$, then $S$ is $1$-dimensional and
$[\tn_+,\tn_+]\subseteq\Ann(S)$. \end{itemize} 
\end{lem}
\begin{proof} 
Denote $\LL=\tn_+/\II((t-a)^n)$ for short. We may consider $S$ as an
$\LL$-module. Let $\Rad(\LL)$ be the radical of $\LL$. Then
$\Rad(\LL)$ is nilpotent and by Proposition 19.1 in \cite{Hu}, we
see that $\Rad(\LL)/\Ann(S)$ is contained in the center of
$\LL/\Ann(S)$, which implies that any element in $\Rad(\LL)$ acts on
$S$ as a scalar. Then we have $\Rad(\LL)\cap[\LL,\LL]\subseteq
\Ann(S)$ since the trace of any element in $[\LL,\LL]$ on $S$ is
$0$. We know that $\II(t-a)/\II((t-a)^n)\subseteq \Rad(\LL)$. By
computing $[\LL,\LL]\cap \II(t-a)$, we deduce that $\Ann(S)$
contains the following subspace
\begin{equation*}\begin{array}{c}
 \sum\limits_{\a\in\Delta_+\setminus\Pi}\g_\a\otimes \gen{t-a}+\sum\limits_{\a\in\Pi}\g_\a\otimes \gen{t(t-a)}+\sum\limits_{\a\in\Delta_+\setminus\{\theta\}}\g_{-\a}\otimes\gen{t(t-a)}\\\\
 +\g_{-\theta}\otimes\gen{t^2(t-a)}+\h\otimes\gen{t(t-a)}.
\end{array}\end{equation*}

(i). If $a\ne0$, combining
with the fact $\II((t-a)^n)\subset \Ann(S)$ we see that
$\II(t-a)\subseteq \Ann(S)$. Since $\tn_+/\II(t-a)\cong \g$ is
simple, we must have $\II(t-a)=\Ann(S)$.

(ii). If $a=0$, we know that $\LL$ is nilpotent and finite-dimensional. So $S$ has to be
$1$-dimensional, and $[\LL,\LL]$ acts trivially on $S$, i.e.,
$[\tn_+,\tn_+]\subseteq \Ann(S)$.
\end{proof}

Take any $\pmb{a}=(a_1,\cdots,a_m)\in\widetilde{\C^m}$ and let $f(t)=\prod\limits_{i=1}^m(t-a_i)$. 
There is a canonical Lie algebra isomorphism
\begin{equation}\label{alg_iso}
\pi_f:\quad\big(\tn_+/\II(tf)\big)\ \rightarrow\
\tn_+/\II(t)\oplus\bigoplus_{i=1}^m\tn_+/\II(t-a_i),
\end{equation}
defined by mapping $x\otimes g(t)$ to $(x\otimes
g(t),\cdots,x\otimes g(t))$. Here and later, we continue to denote
by $x\otimes g(t)$ its image in $\tn_+/\II$ for any $x\otimes
g(t)\in \tn_+$ and an ideal $\II\subseteq\tn_+$ for convenience.
Using \eqref{canonical_iso} we further have the Lie algebra
isomorphism
\begin{equation}\label{alg_iso2}\begin{array}{c}
\tn_+/\II(tf)\ \rightarrow\
(\tn_+/\II(t))\oplus\bigoplus_{i=1}^m\g,\\\\ x\ot g(t)\mapsto (x\ot
g(t),g(a_1)x,...,g(a_m)x), \forall
\,\,x\in\g,\ g(t)\in\C[t].\end{array}\end{equation}

Given any Lie algebra homomorphism $\eta: \tn_+\rightarrow \C$. We
can define a $1$-dimensional $\tn_+$-module $\C v_{\eta}$ by
$xv_{\eta}=\eta(x)v_\eta$ for all $x\in\tn_+$. Note that
$\eta([\tn_+,\tn_+])=0$ and hence $\eta$ is completely determined by
its values on $\sum\limits_{\a\in\Pi}\g_\a+\g_{-\theta}\otimes t$. For any
$\g$-modules $L_1,...,L_m$, using the isomorphism in
\eqref{alg_iso2} and noticing that $\II(t)\subseteq[\tn_+,\tn_+]$,
we can make $ L_1\ot\cdots\ot L_m\ot \C v_\eta$ into an
$\tn_+/\II(f)$-module via:
\begin{equation*}\begin{array}{c}
 (x\ot g(t))v_1\ot\cdots \ot v_m\ot v_\eta=\eta(x\ot g(t))v_1\cdots\ot v_m\ot v_\eta\\\\
 +\sum_{i=1}^mg(a_i)v_1\ot\cdots \ot xv_i\ot\cdots \ot v_m\ot v_\eta,
\end{array}\end{equation*}
for all $v_i\in L_i$ and $x\ot g(t)\in\tn_+$.
We denote the resulted module by $S(\eta; L_1,\cdots,L_m;
a_1,\cdots,a_m)$.
Note that $S(\eta; L_1,\cdots,L_m; a_1,\cdots,a_m)$ can be naturally
viewed as an $\tn_+$-module. Here we point out that when $\eta=0$,
then $\C v_\eta$ is a trivial $\tn_+$-module and $\C v_\eta$ makes
no contribution to the $\tn_+$-module $S(\eta; L_1,\cdots,L_m;
a_1,\cdots,a_m)$.

Let $L(\l_i)$ be the irreducible $\g$-module with highest weight
$\l_i\in \h^*$ for $ i=1,\cdots,m$. Then we denote $S(\eta;
L(\l_1),\cdots,L(\l_m); a_1,\cdots,a_m)$ by $S(\eta,
{\pmb{\lambda}}, \pmb{a})$ for short, where
$\pmb{\lambda}=(\l_1,\cdots,\l_m)$
and $\pmb{a}=(a_1,\cdots,a_m)$. 
It is not hard to prove that $S(\eta, {\pmb{\lambda}}, \pmb{a})$ is
an irreducible $\tilde{\n}_+$-module as well as an irreducible
$\tilde{\n}_+/\II(tf)$-module. If all $\l_1,\cdots,\l_m$ are
dominant, then $S(\eta,\pmb{\l},\pmb{a})$ is finite dimensional. We
also note that when some $\l_i=0$, this $L(\l_i)$ is a
$1$-dimensional trivial $\g$-module and hence makes no contribution
to the $\tn_+$-module $S(\eta,\pl,\pa)$.

\begin{lem}\label{I(f)} Let  ${f}(t)=\prod \limits_{i=1}^m(t-a_i)$ and $\tilde
f(t)=\prod \limits_{i=0}^m(t-a_i)^{n_i}$, where
$\pa=(a_1,\cdots,a_m)\in\widetilde{\C^m}$, $a_0=0$, $n_0\in\Z_+$ and other $n_i\in\N$.
Let $S$ be an irreducible finite dimensional $\tn_+$-module with
$\II(\tilde f)S=0$. Then $S\cong S(\eta, \pmb{\lambda}, \pmb{a})$
for some Lie algebra homomorphism $\eta:\tn_+\rightarrow \C$ and
$\pmb{\lambda}=(\l_1,\cdots,\l_m)$ with each $\l_i\in\h^*$ being
dominant. Moreover, \begin{itemize} \item[(i).]if $n_0=0$ then $\eta=0$ and
$\II({f})\subseteq\Ann(S)$;  \item[(ii).] if $n_0\neq0$ we have
$$\sum_{\a\in\Delta_+\setminus\Pi}\g_\a\otimes \gen{{f}}+\sum_{\a\in\Pi}\g_\a\otimes \gen{t{f}}
+\sum_{\a\in\Delta_+\setminus\{\theta\}}\g_{-\a}\otimes\gen{t{f}}$$
$$ +\g_{-\theta}\otimes\gen{t^2{f}}
+\h\otimes\gen{t{f}}\subseteq\Ann(S).$$\end{itemize}
\end{lem}

\begin{proof} Denote $\LL=\tn_+/\II(\tilde f)$. We regard $S$ as an $\LL$-module. In a similar Lie algebra isomorphism as in
(\ref{alg_iso}): $$\tilde{\pi}_f:\quad\tn_+/\II(\tilde f)\
\rightarrow\ \bigoplus_{i=0}^m\tn_+/\II((t-a_i)^{n_i}),$$ let
$\LL_i=\tn_+/\II((t-a_i)^{n_i})$. Hence $S$ can be viewed as an
$\LL_i$-module for each $i$. We choose a nonzero irreducible
$\LL_i$-submodule of $S$, say $L_i$.

First suppose $i\neq0$. 
By Lemma \ref{I(t-a)}, we see that   $\II(t-a_i)
L_i=0$. So $L_i$ is an irreducible module over the simple Lie algebra $\g\cong\tn_+/\II(t-a_i)$.
We see that $L_i\cong L(\l_i)$ for some weight of $\g$ with respect
to $\h$, i.e.,$(x\ot t^k)v=a_i^kxv$ for any $v\in L_i$, $x\in\g$ and
$k\in \N$. Moreover, $\l_i$ is dominant, and $\l_i=0$ if further $L_i$ is
trivial.

Now suppose that $i=0$ and $n_0\ne0$. Then $\LL_0$ is nilpotent,
furthermore $L_0$ is $1$-dimensional and there exist a Lie algebra
homomorphism $\eta_0: \LL_0\rightarrow \C$ and $w\in L_0$ such that
$xw=\eta_0(x)w$ for all $x\in\LL_0$. Again by Lemma \ref{I(t-a)}, we
get $[\tn_+,\tn_+]L_0=0$. Recalling that
$\II(t)\subseteq[\tn_+,\tn_+]$, so we obtain $\II(t)L_0=0$.

If $n_0=0$, then  $\LL_0=0$ and hence $\eta=0$.

From \cite[Lemma 2.7]{Li} we know that $S\cong S(\eta, \pmb{\lambda}, \pmb{a})$. Part (i) follows easily. For Part (ii), computing the left hand side of  $[\tn_+,\tn_+]\cap \II(f)\subset \Ann(S)$ we   deduce the statement in (ii).
The lemma is true.
\end{proof}

\section{Classification  of locally finite irreducible modules} 

In this section we will obtain the complete classification of
irreducible  $\tg$-modules on which each root vector in $\tn_+$
acts on the modules locally finitely.

Let $\pl=(\l_1,\cdots,\l_m)\in (\h^*)^m$, $\pa=(a_1,a_2,...,a_m)\in
\widetilde{\C^m}$, and let $\eta: \tn_+\rightarrow \C$ be a Lie
algebra homomorphism. We will denote by $W(\eta)$   an irreducible
Whittaker $\hg$-module with Whittaker function $\eta$ if $\eta\ne0$.
For any $\gamma\in\hat{\h}^*$, we  denote by $\hat V(\gamma)$ the
irreducible highest weight $\hg$-module  with highest weight
$\gamma$.

Let $S=S(\eta,\pl,\pa)$ be as defined in Sect.4, and let $f_{\pa}= \prod \limits_{i=1}^m(t-a_i)$. Then
$\II(tf_{\pa})S=0$, and $\II(f_{\pa})S=0$ if additionally $\eta=0$. 
Recall from \eqref{alg_iso} that
$$\aligned \tn_+/\II(tf_{\pa})\cong& (\tn_+/\II(t))\oplus(\tn_+/\II(f_a))\\
\cong&(\tn/\II(t))\oplus\bigoplus_{i=1}^m\tn_+/\II(t-a_i).\endaligned$$ 
Let $\pi_i$ be the projection from $\tn_+/\II(f_{\pa})$ to
$\g^{(i)}=\tn_+/\II(t-a_i)$ which is isomorphic to $\g$ for all
$i=1,\cdots,m$. We identify $\tn_+/\II(f_{\pa})$, which can be
viewed as an ideal of $\tn_+/\II(tf_{\pa})$, with $\g^m$. Then
regarded as an $\tn_+/\II(f_{\pa})$-module, $S$ is an irreducible
highest weight
module with highest weight $\pl\in(\h^*)^m$. 

We consider the induced $\tg$-module
$$\tM(\eta,\pl,\pa)=\Ind_{\tn_+}^{\tg}S(\eta,\pl,\pa)=U(\tg)\ot_{U(\tn_+)}S(\eta,\pl,\pa),$$
and  the induced $\hg$-module
$$\hM(\eta,\pl,\pa)=\Ind_{\tn_+}^{\hg}S(\eta,\pl,\pa)=U(\hg)\ot_{U(\tn_+)}S(\eta,\pl,\pa).$$
Clearly, $\tM(\eta,\pl,\pa)=\hM(\eta,\pl,\pa)[d]=\C[d]\ot
\hM(\eta,\pl,\pa)$. Let $\C u_\eta$ be the $1$-dimensional
$\tn_+$-module such that $x u_\eta=\eta(x)u_\eta$ for all $x\in
\tn_+$. Then we have the induced $\hg$-module
$$\hW(\eta)=\Ind_{\tn_+}^{\hg}\C u_\eta.$$
If $\eta\neq0$,  $\hW(\eta)$ is the universal Whittaker module with
Whittaker function $\eta$.
Recall that the irreducible $\hg$-module $E(\pl,\pa)$ 
is defined in Section 3 (before Theorem \ref{tensor}).

\begin{lem}\label{phi-iso} As $\hg$-modules we have
$\hM(\eta,\pl,\pa)\cong E(\pl,\pa)\ot \hW(\eta)$.
\end{lem}

\begin{proof} Regarding $E( \pl,\pa)$ as an $\tn_+$-module,
we have the canonical $\tn_+$-module isomorphism $\phi:
S(0,\pl,\pa)\rightarrow E(\pl,\pa)$.   We get   the canonical
$\tn_+$-module homomorphism
$$S(\eta,\pl,\pa)= S(0,\pl,\pa)\ot v_\eta\longrightarrow
 E(\pl,\pa)\ot \hW(\eta),$$ $$  v\ot v_\eta\mapsto \phi(v)\ot
u_\eta,\ \forall\ v\in S(0,\pl,\pa),$$
whose restriction is  the canonical
$\tn_+$-module isomorphism
$$S(\eta,\pl,\pa) \longrightarrow
 E(\pl,\pa)\ot u_\eta.$$
From the universal property we obtain an induced $\hg$-module homomorphism
$$\widehat{\phi}:
\hM(\eta,\pl,\pa)=\Ind_{\tn_+}^{\hg}S(\eta,\pl,\pa)\hskip 3cm$$
$$\hskip 3cm \to
 E(\pl,\pa)\ot \hW(\eta)= E(\pl,\pa)\ot (\Ind_{\tn_+}^{\hg}\C u_\eta).$$

Note that $\hM(\eta,\pl,\pa)= U(\hb_-)\ot S(\eta,\pl,\pa)$ and $$
E(\pl,\pa)\ot \hW(\eta)=   E(\pl,\pa) \ot U(\hb_-)\ot\C u_\eta$$ as
vector spaces, where $\hb_-=\hh+\tn_-$. Take a basis  $\{
b_1,b_2,\cdots\}$ of $\hb_-$.  We have a PBW basis
$$\{b_n^{r_n}\cdots b_1^{r_1}\ |\ n\in\Z_+, r_1,\dots, r_n\in\Z_+\}$$
for $U(\hb_-)$. For any $l\in\Z_+$, we denote
$$U(\hb_-)_l=\sum_{\begin{matrix} n, r_i\in\Z_+,\\  r_1+\cdots+r_n=l\end{matrix} }\C b_n^{r_n}\cdots b_1^{r_1}.$$
By induction on $l\in\Z_+$, we can easily deduce the following\vskip5pt

\noindent{\bf Claim 1.} For any $x\in U(\hb_-)_l, l\in\Z_+$ and
$v\in S(0,\pl,\pa)$, $$\hphi\big(x(v\ot v_\eta)\big)=x(\phi(v)\ot
u_\eta)\equiv\phi(v)\ot (x u_\eta) \hskip -7pt \mod E(\pl,\pa)\ot
\sum_{i=0}^{l-1} U(\hb_-)_iu_\eta.$$

Noticing that $\phi$ is an isomorphism, we see $ E(\pl,\pa)\ot
\sum \limits_{i=0}^{l-1}U(\hb_-)_iu_\eta\subseteq \hphi(\hM(\eta,\pl,\pa))$
implies $ E(\pl,\pa)\ot \sum \limits_{i=0}^{l}U(\hb_-)_iu_\eta \subseteq
\hphi(\hM(\eta,\pl,\pa))$ from the claim. Then by induction we can
obtain that $ E(\pl,\pa)\ot \hW(\eta)\subseteq
\hphi(\hM(\eta,\pl,\pa))$, that is, $\hphi$ is surjective.

Now we prove that $\hphi$ is injective. For any nonzero $w\in
\hM(\eta,\pl,\pa)$, we write
$$w=\sum_{i=0}^kx_i(v_i\ot v_\eta)\mod  \sum_{i=0}^{l-1}U(\hb_-)_i( S(0,\pl,\pa)\ot u_\eta),$$
where $x_i\in U(\hb_-)_l$ and $v_i\in S(0,\pl,\pa)$. We may assume
that $v_i$'s are linearly independent and each $x_i\ne0$. Using
Claim 1, we see that $$\hphi(w)\equiv\sum_{i=0}^k \phi(v_i)\ot (x_i
u_\eta)  \mod E(\pl,\pa)\ot \sum_{i=0}^{l-1}U(\hb_-)_iu_\eta.$$
which is nonzero since  $\phi(v_i)$'s are linearly independent and
each $x_i\ne0$. Thus $\hphi$ is injective. We conclude that $\hphi$
is an isomorphism.
\end{proof}

As a result of Lemma \ref{phi-iso}, we have a $\tg$-module
isomorphism
$$\tphi: \tM(\eta,\pl,\pa)=\Ind_{\hg}^{\tg}\hM(\eta,\pl,\pa)\longrightarrow
\Ind_{\hg}^{\tg}\big(E(\pl,\pa)\ot \hW(\eta)\big).$$ Note that
$\Ind_{\hg}^{\tg}\left(E(\pl,\pa)\ot \hW(\eta)\right)=
(E(\pl,\pa)\ot \hW(\eta))[d]$. By Theorem \ref{E[d]}, any
$\tg$-submodule of $(E(\pl,\pa) \ot \hW(\eta))[d]$ is of the form $(
E(\pl,\pa)\ot N)[d]$ for some $\hg$-submodule $N$ of $\hW(\eta)$.
Thus any irreducible $\tg$-quotient module of $(\hW(\eta)\ot
E(\pl,\pa))[d]$ is of the form $(W(\eta)\ot E(\pl,\pa))[d]$ for some
irreducible $\hg$-quotient module $W(\eta)$ of $\hW(\eta)$. Now we
can characterize all irreducible $\hg$-quotient modules   $W(\eta)$ of
$\hW(\eta)$.

\begin{lem}\label{quotient-hW} Let $\eta: \tn_+\rightarrow \C$ be a Lie algebra homomorphism,
and $\hW(\eta)$ be the universal Whittaker module over $\hg$.
\enumerate\item  If $\eta\neq0$, then any irreducible quotient
module of $\hW(\eta)$  is an
 irreducible  Whittaker $\hg$-module of type $\eta$.
\item If $\eta=0$, then any irreducible quotient module of $\hW(\eta)$
is an irreducible highest weight module $\hat V(\gamma)$ over $\hg$
with some highest weight $\gamma\in \hat \h^*$.\endenumerate
\end{lem}

\begin{proof} Part (1) is clear. Now suppose $\eta=0$, and $W$ is an irreducible quotient $\hg$-module of
$\hW(0)$. 
Consider the induced $\tg$-module $\widetilde
W=W[d]=\text{ind}_{\hg}^{\tg}W$.

If $\widetilde W$ is irreducible, using Theorem 1 in \cite{MZ} we
know that $\widetilde W$ is an irreducible  highest weight
$\tg$-module which is a contradiction to the fact that the action of
$d$ is free on $\widetilde W$.

So $\widetilde W$ is not irreducible. Let $Y$ be a nonzero proper
$\tg$-submodule of $\widetilde W$. It is obvious that $Y\cap W=0$,
since otherwise $$\tW=U(\tg)W=U(\tg)U(\hg)(Y\cap W)\subseteq Y.$$ Because $W$ is irreducible over $\hg$,
 $\widetilde W$ has a maximal submodule over $\tg$.
Take
a maximal $\tg$-submodule $X$ of $\widetilde W$ and choose any
$$w=d^n\ot w_n+d^{n-1}\ot w_{n-1}+...+w_0\in X, \ w_i\in W,\
w_n\ne0$$ such that $n$ is minimal. For any $x\in U(\hg)$, $k\in\Z$
and $l\in\Z_+$, by computing $d^lxw\in X$ and noticing that $W$ is
an irreducible $\hg$-module, we see that for any $v\in W$ there is
an element $d^{n+l}\otimes v+\sum\limits_{i=0}^{n+l-1}d^i\otimes v_i\in X$
for some $v_i\in W$, and hence $X+W\oplus (d\ot W) \oplus ...\oplus
(d^n\ot W)=\tW$. In particular,
 $$\widetilde W/X\cong \frac{W\oplus (d\ot W) \oplus ...\oplus (d^n\ot W)}{X\cap \big(W\oplus (d\ot W) \oplus ...\oplus (d^n\ot W)\big)}$$
is a nontrivial irreducible $\tg$-module. Using Theorem 1 in
\cite{MZ} we know that $n=1$ and
 $\widetilde W/X$ is a highest weight $\tg$-module. If we consider  $\widetilde W/X$ as a $\hg$-module we have  $\widetilde W/X\cong W$. Hence $W$ is a highest
 weight $\hg$-module. Part (2) follows.
\end{proof}

Using the above two lemmas, we can give the characterization of
irreducible $\tg$-modules such that the actions of elements $\tn_+$
are locally finite.

\begin{thm}\label{char} Let $V$ be an irreducible $\tg$-module on which
the action of each root vector in $\tn_+$ is locally finite. Then
$V$ is isomorphic to one of the following $\tg$-module:

\enumerate \item an irreducible highest weight module $\widetilde
V(\gamma)$ for some $\gamma\in \th^*$;
\item an irreducible Whittaker module $\widetilde W(\eta)$ for a nonzero Lie algebra homomorphism $\eta:\tn_+\to \C$;
\item $(\hat V(\gamma)\ot E(\pl,\pa))[d]$ for some
$\gamma\in\hat{\h}^*$, $\pa\in\widetilde{\C^m}$ and $\pl\in
(\h^*)^m\setminus\{0\}$ with all $\l_i$ dominant, where $\hat
V(\gamma)$ is the irreducible highest weight module over $\hg$;
\item $(W(\eta)\ot E(\pl,\pa))[d]$ for a nonzero Lie algebra
homomorphism $\eta:\tn_+\to \C$, $\pa\in\widetilde{\C^m}$, and
$\pl\in (\h^*)^m\setminus\{0\}$ with all $\l_i$
dominant.\endenumerate
\end{thm}

\begin{proof} By Theorem \ref{local finite}, $V$ has
a finite-dimensional $\tn_+$-submodule and hence an irreducible
finite-dimensional $\tn_+$-submodule, say $S$. If this $S$ is
one-dimensional, we see that $V$ is a highest weigh $\tg$-module or
an irreducible Whittaker $\tg$-module by Lemma \ref{quotient-hW}.
Otherwise, by Theorem 4.1 and Lemmas 4.2 and 4.3, this irreducible
$\tn_+$-submodule is isomorphic to $S(\eta,\pl,\pa)$ for a  Lie
algebra homomorphism $\eta:\tn_+\to \C$, $\pa\in\widetilde{\C^m}$,
and $\pl\in (\h^*)^m$ with all $\l_i$ dominant and at least one
nonzero. Then we know that $V$ is an irreducible quotient of the
$\tg$-module $\widetilde
M(\eta,\pl,\pa)=\Ind_{\tn_+}^{\tg}S(\eta,\pl,\pa)=(E(\pl,\pa)\ot
\hW(\eta))[d].$ The theorem follows from Theorem \ref{E ot L[d]} and
Lemmas \ref{phi-iso}, \ref{quotient-hW}.
\end{proof}

\section{Isomorphism classes of irreducible $\tg$-modules}

In this section we determine the necessary and sufficient conditions
for two irreducible modules of the form $(E(\pl,\pa)\ot \hat
V(\gamma))[d]$ or $(E(\pl,\pa)\ot W(\eta))[d]$ to be isomorphic,
where $\eta:\tn_+\to \C$ is  a nonzero Lie algebra homomorphism,
$\pa\in\widetilde{\C^m}$, $\pl=(\lambda_1,...,\lambda_m)\in
(\h^*)^m\setminus\{0\}$ and $\gamma\in\hat{\h}^*$. We have assumed
that $\hat V(\gamma)$ is the irreducible highest weight $\hg$-module
of highest weight $\gamma$ and $W(\eta)$ is an irreducible Whittaker
$\hg$-module.

\begin{lem}\label{iso_M[d]} Let $M$ and $M'$ be irreducible $\hg$-modules. Then
$M[d]\cong M'[d]$ as $\tg$-modules if and only if $M\cong M'$ as
$\hg$-modules.
\end{lem}
\begin{proof}  The sufficiency is clear. We need only to prove the
necessity.
Let $\widetilde \phi: M[d]\rightarrow M'[d]$ be a $\tg$-module
isomorphism. We can define a map $\phi: M\rightarrow M'$ as follows.
Take any nonzero $u\in M$, there exist unique $l_u\in\Z_+$ and
$v_i\in M'$ with $v_{l_u}\neq 0$ such that $\widetilde
\phi(u)=\sum \limits_{i=0}^{l_u}d^i v_i$. Define $\phi(u)=v_{l_u}$. Note
that for any $x\in \hg$, we have
\begin{equation}\label{tphi}
\widetilde\phi(xu)=x\widetilde\phi(u)=\sum_{i=0}^{l_u}x(d^iv_i)=d^{l_u}xv_{l_u}+\sum_{i=0}^{l_u-1}d^ixv_i+\sum_{i=0}^{l_u}[x,
d^i]v_i. \end{equation} This together with the fact $M=\UU(\hg)u$
implies that $l_{u'}\leq l_u$ for all $u'\in M$ and hence
$l_u=l_{u'}$ for all nonzero $u,u'\in M$. Then the linearity of
$\varphi_i$ follows from its definition and \eqref{tphi} yields
$\phi(xu)=xv_{l_u}=x\phi(u)$. As a result, $\phi$ is a nonzero
$\hg$-module homomorphism and hence an isomorphism between
$\hg$-modules $M$ and $M'$.
\end{proof}

\begin{lem}\label{iso_E ot L} Let $\pl\in (\h^*)^m\setminus\{0\}$, 
$\pl'\in (\h^*)^{m'}$, $\pa\in \widetilde{\C^m}$,  $\pa'\in
\widetilde{\C^{m'}}$, and let $L, L'$ be $\hg$-modules such that
for any $w\in L$ and $w'\in L'$ we have $(\g\ot t^{k}\C[t])w=(\g\ot
t^{k}\C[t])w'=0$ for some $k\in\N$ depending on $w, w'$. Then
$E(\pl,\pa)\ot L\cong E(\pl',\pa')\ot L'$ if and only if
$E(\pl,\pa)\simeq E(\pl',\pa')$ and $L\cong L'$.
\end{lem}

\begin{proof} We know that both $E(\pl,\pa)$ and $E(\pl',\pa')$ are irreducible $\hg$-modules.
Let $\widetilde \phi: E(\pl,\pa)\ot L\rightarrow E(\pl',\pa')\ot L'$ be a $\hg$-module
isomorphism. Fix  nonzero $v\in E(\pl,\pa)$ and nonzero  $w\in L$. Let
$$\widetilde{\phi}(v\ot w)=\sum_{i=1}^sv_i\ot w_i,$$
where $v_i\in E(\pl',\pa'), w_i\in L'$ such that $w_i$'s are linearly independent. We may also assume that each $v_i\ne0$.
Choose any $k\in\N$ such that $(\g\ot
t^{k}\C[t])w=(\g\ot t^{k}\C[t])w_i=0$ for all $i$.
Note that
\begin{equation}\label{tphi2}
\widetilde{\phi}(xv\ot w)=x\widetilde{\phi}(v\ot w)=\sum_{i=1}^sxv_i\ot w_i,\ \forall\ x\in \g\ot t^{k}\C[t].
\end{equation}
Combining this with the fact that $E(\pl,\pa)=U(\hg)v$ and using
Lemma \ref{iso}, we see that, for any
$u\in E(\pl,\pa)$,  $$\widetilde{\phi}(u\ot
w)=\sum_{i=1}^su_i\ot w_i$$ where $u_i\in E(\pl',\pa')$. Now we can define nonzero vector space
homomorphisms $\phi_i:E(\pl,\pa)\to E(\pl',\pa')$ by $\phi_i(u)=u_i$
for all $u\in E(\pl,\pa)$ and we have $$\widetilde{\phi}(u\ot
w)=\sum_{i=1}^s\phi_i(u)\ot w_i,\ \forall\ u\in E(\pl,\pa).$$ Note
that \eqref{tphi2} with $v$ replaced by $u$ indicates that
$\phi_i(xu)=x\phi_i(u)$ for $x\in\g\ot t^{k}\C[t]$ and $u\in
E(\pl,\pa)$. The irreducibility of $E(\pl', \pa')$ and Lemma
\ref{iso} shows that all $\phi_i$ are surjective, i.e.,
$\phi_i(E(\pl,\pa))=E(\pl',\pa')$.

Let  $f_{\pa}(t)=\prod \limits_{j=1}^m(t-a_j)$,
$f_{\pl,\pa}(t)=\prod \limits_{j=1,\l_j\neq 0}^m(t-a_j)$ and similar for
$f_{\pa'}$ and $f_{\pl',\pa'}$. Now we have
$0=\phi_i(xu)=x\phi_i(u)$ for all $u\in E(\pl,\pa)$ and $x\in
f_{\pl,\pa}(t)t^k\C[t]$. Since $\varphi_i$ is surjective, we get
$xE(\pl',\pa')=0$ for all $x\in f_{\pl,\pa}(t)t^k\C[t]$, yielding
from Remark \ref{remark} that
$$f_{\pl,\pa}(t)t^k\in\Ann(E(\pl',\pa'))=f_{\pl',\pa'}(t)\C[t^{\pm1}]\oplus\C
K.$$ Thus we get $f_{\pl',\pa'}$ divides $f_{\pl,\pa}$ and
symmetrically $f_{\pl,\pa}$ divides $f_{\pl,\pa'}$. In particular,
$f_{\pl,\pa}=f_{\pl,\pa'}$, $\Ann(E(\pl,\pa))=\Ann(E(\pl',\pa'))$
and $\phi_i(xu)=0=x\phi_i(u)$ for $x\in\g\ot
f_{\pl,\pa}\C[t^{\pm1}]$.

Since $\g\ot f_{\pl,\pa}\C[t^{\pm1}]+\g\ot t^{k}\C[t]=\g\ot
\C[t^{\pm1}]$, we obtain that $\phi_i(xu)=x\phi_i(u)$ for
$x\in\g\ot\C[t^{\pm1}]$. Hence each $\phi_i$ is a $\hg$-module
homomorphism. Since $E(\pl,\pa)$ is a simple modules, $\phi_i\neq0$
is also injective. Consequently, each $\phi_i$ is a $\hg$-module
isomorphism. Thus $E(\pl,\pa)\simeq E(\pl',\pa')$.

It is well known that any endomorphism of a simple module over a
countably generated associative $\C$-algebra
  is a scalar (Proposition 2.6.5 and Corollary 2.6.6 in [D]).
  There
exist $a_i\in \C^*$ such that $\phi_i=a_i\phi_1$. Let
$w'=\sum_{i=1}^sa_iw_i$. Then
$$\widetilde{\phi}(u\ot w)=\phi_1(u)\ot w', \forall\ u\in E(\pl,\pa).$$
Applying $x\in\hg$ to both sides we deduce that
\begin{equation*}\begin{split}
\widetilde{\phi}(x(u\ot w))=&\widetilde{\phi}(xu\ot
w)+\widetilde{\phi}(u\ot x w)\\ =&{\phi}_1(xu)\ot
w'+\widetilde{\phi}(u\ot x w)\\ =&(x\phi_1(u))\ot
w'+\widetilde{\phi}(u\ot x w),\\
x\widetilde{\phi}(u\ot w)=&(x\phi_1(u))\ot w'+\phi_1(u)\ot xw'
\end{split}\end{equation*}
i.e., $\widetilde{\phi}(u\ot x w)=\phi_1(u)\ot xw'$. Thus we have a
vector space homomorphism $\phi': L\to L'$ such that
$$\widetilde{\phi}(u\ot w)={\phi}_1(u)\ot \phi'(w), \forall\ u\in
E(\pl,\pa),\,\, w\in L.$$ Applying $x\in\hg$ to both sides we deduce
that  $\phi'(x w)=x\phi'(w)$ for all $x\in\hg$ and all $w\in L$.
That is, $\phi': L\rightarrow L'$ is a nonzero $\hg$-module
homomorphism. Since $\widetilde{\varphi}$ is an isomorphism of
$\tg$-modules, we obtain that $\phi'$ is an isomorphism of
$\hg$-modules, i.e., $L\simeq L'$.\end{proof}

Note that the necessary and sufficient conditions for the
isomorphism $E(\pl,\pa)\simeq E(\pl',\pa')$ is well-known, see
Theorem 1.3(d) in \cite{CP3}. Combine Lemma \ref{iso_M[d]}, Lemma
\ref{iso_E ot L} and highest weight representation theory of $\hg$,
we can deduce the isomorphisms among the irreducible $\tg$-modules
$(E(\pl,\pa)\ot \hat V(\gamma))[d]$ and $(E(\pl,\pa)\ot
W(\eta))[d]$.

\begin{thm} Let $\pl \in (\h^*)^m\setminus\{0\}$, $\pa\in
\widetilde{\C^m}$, $\pl'\in (\h^*)^{m'}$, $\pa'\in
\widetilde{\C^{m'}}$, $\gamma, \gamma'\in\hh^*$ and $\eta, \eta':
\tn_+\rightarrow \C$ are nonzero Lie algebra homomorphisms.

 \enumerate \item $(E(\pl,\pa)\ot \hat V(\gamma))[d]$ and
$(E(\pl',\pa')\ot W(\eta))[d]$ can never be isomorphic;
\item
$(E(\pl,\pa)\ot \hat V(\gamma))[d]\simeq (E(\pl',\pa')\ot
V(\gamma'))[d]$ as $\tg$-modules if and only if $E(\pl,\pa)\simeq
E(\pl',\pa')$ as $\hg$-modules and $\gamma=\gamma'$; \item
$(E(\pl,\pa)\ot W(\eta))[d]\simeq (E(\pl',\pa')\ot W(\eta'))[d]$ as
$\tg$-modules if and only if $E(\pl,\pa)\simeq E(\pl',\pa')$ and
$W(\eta)\cong W(\eta')$ as $\hg$-modules.\endenumerate
\end{thm}

\section{Tensor products of highest weight modules  and loop modules  over untwisted affine Kac-Moody algebras}

In this last section we will  use the ``shifting technique" to obtain a
necessary and sufficient condition for the tensor product of
irreducible integrable loop modules and irreducible integrable highest weight
modules over untwisted affine Kac-Moody algebras $\tg$ to be simple. This
problem was started by Chari and Pressley in [CP3], and studied by
them and other authors, for example  [Ad1-3]. We will start
with a slightly more general setting.

For any $\pa\in\widetilde{\C^m}$, $\pl=(\lambda_1,\lambda_1,\cdots  , \lambda_m) \in (\h^*)^m$, we have the
irreducible $\hg$-module  $E(\pl,\pa)$. For any $b\in\C$ let us
recall from \cite{E2} or \cite{CP3} the $\tg$-modules $E(\pl,\pa,
b)=E(\pl,\pa)\otimes \C[t,t^{-1}]$ with actions
\begin{equation}\label{E(b)}\aligned (x\otimes t^k)(v\otimes t^n)=((x\otimes t^k)v)\otimes t^{n+k},
\\K(v\otimes t^n)=0,\,\,\, d(v\otimes t^n)=(b+n)(v\otimes t^n),\endaligned\end{equation}
for all $k,n\in\Z$, $x\in\g$, $v\in E(\pl,\pa)$.  Note
that the $\tg$-modules $E(\pl,\pa, b)$ are integrable weight
$\tg$-modules if all $\lambda_i$ are dominant.

Let us recall from \cite{CP3} the graded associative algebra
homomorphism $\chi: U(\h\otimes \C[t,t^{-1}])\to \C[t,t^{-1}]$
defined by extending
\begin{equation}\label{chi}\chi(h\otimes t^n)=\Big(\sum_{i=1}^m\l_i(h)a_i^n\Big)t^n.\end{equation}
 By Theorem 1.3 (c) of \cite{CP3},  we see that  $E(\pl, \pa, b)$ is
irreducible,  if and only if  $\chi$
is surjective. In other words, for all $k\in\Z$ there exist $h_k\in
U(\h\otimes \C[t,t^{-1}])$ such that $[d, h_k]=kh_k$ and
$\chi(h_k)\neq0$. Although the proof in  \cite{CP3} is for finite dimensional $E(\pl,\pa)$, it is
valid even for any infinite dimensional $E(\pl,\pa)$.

Let $L$ be an irreducible $\tg$-module
such that \begin{equation} \label{L mod} L=\sum_{i\in\Z_+}L_{-i},\
{\rm where}\ L_{-i}=\{ v\in L\ |\ dv=(\zeta-i)v\}
\end{equation}
for some $\zeta\in\C$ with $L_{0}\neq0$. We call $\zeta$ the highest
degree of $L$. Note that irreducible highest weight $\tg$-modules
are such modules. There are certainly a lot of other such
irreducible modules.
 Similarly we have
$$U(\tg)=\sum_{i\in\Z}U(\tg)_{i},\ {\rm where}\ U(\tg)_{i}=\{ x\in U(\tg)\ |\
[d,x]=ix\}.$$ Any vector $x\in U(\tg)_i$ is called homogeneous of
degree $i$, denoted by $\deg(x)=i$. Clearly $U(\tg)_iL_{j}\subseteq
L_{i+j}$, where we make the convention that $L_{j}=0$ for $j\in\N$.
From now on we fix one such module $L$.

\begin{lem}\label{L irre} The above module $L$ is also irreducible as a $\hg$-module.
\end{lem}

\begin{proof} Let $W$ be a nonzero $\hg$-submodule of $L$. Take any nonzero $w\in W$.
Write $w=\sum\limits_{i=0}^nu_i$ with $u_i\in L_{-i}$ and $u_n\ne0$. Since
$L$ is simple as a $\tg$-module, there exists $x_1,\cdots,x_s\in \g$
and $j_1,\cdots, j_s\in \Z_+$ with  $j_1+\cdots + j_s=n$  such that
$w_0=(x_1\ot t^{j_1})\cdots (x_s\ot t^{j_s})w\in L_0\setminus\{0\}$.
Then the $\hg$-submodule generated by $w_0$ is a $\tg$-submodule of
$L$. Hence $W=L$. Consequently, $L$ is  irreducible as a
$\hg$-module.
\end{proof}

We will use the ``shifting technique" established in \cite{CGZ} to study the tensor product
  $\tg$-module
$E(\pl,\pa, b)\ot L$.
Define a new action of $\tg$ on the vector space $E(\pl,\pa)\ot
L\ot\C[t^{\pm1}]$ by
\begin{equation}\label{E(B, L)}\aligned (x\otimes t^k)\big(v\otimes  u\otimes t^n\big)=&\Big((x\otimes t^k )v\otimes  u+ v\otimes (x\otimes t^k) u\Big)\otimes t^{n+k},
\\
d\big(v\otimes u\otimes t^n\big)=&(\zeta+b+n)(v\otimes u\otimes t^n),\\ K(v\otimes u\otimes t^n)=&v\otimes Ku\otimes t^n,\endaligned\end{equation}
where $b\in\C$, $ u\in L$, $v\in E(\pl,\pa)$, $x\in\g$ and $k,
n\in\Z$. Denote this $\tg$-module as $\mathcal{E}(\pl,\pa, b, L)$.
On the other hand, we can form the tensor product $\tg$-module
$E(\pl,\pa, b)\ot L$. Then we have the $\tg$-module isomorphism
$\mathcal{E}(\pl,\pa, b, L)\cong E(\pl,\pa, b)\otimes L$ given by:
\begin{equation*}\begin{split}
v\ot u\ot t^n&\mapsto (v\ot t^{n+i})\ot u,\ \forall\ u\in L_{-i}.
\end{split}\end{equation*}

 In the rest of this section, we will consider the
module $\mathcal{E}(\pl,\pa, b, L)$ instead of $E(\pl,\pa, b)\otimes
L$, to deduce their irreducibilities. We first prove a property of
the function $\chi$ defined in \cite{CP3}. For any $p, k\in\N$,
define
$$U_k\big(\h\ot(t^p\C[t])\big)=\{x\in  U\big(\h\ot(t^p\C[t])\big)\,\,|\,\, [d, x]=kx\}.$$
We further have

\begin{lem}\label{cal E}
Suppose that  $\chi$ is onto.  Then for any $p\in\N$, we have
$\chi\big(U_k(\h\ot(t^p\C[t]))\big)\ne0$ for all sufficiently large $k\in\N.$ 
\end{lem}

\begin{proof} Let $S_p=\{k\in\N\,\,|\,\,\chi(U_k(\h\ot(t^p\C[t]))\ne0\}$, and $\langle
S_p\rangle $  be the subgroup of $\Z$ generated by $S_p$. It is
enough to show that $\langle S_p\rangle = \Z$. Otherwise suppose
that $\langle S_p\rangle =n_1\Z\ne \Z$ where $n_1>1$.  Then
$\chi(\h\ot t^{q+n_1k})=0$ for all $k\ge p$ and all $q=1, 2,...,
n_1-1,$ i.e.,
$$\sum_{i=1}^m\l_i(h)a_i^q(a_i^{n_1})^k=0, \,\,\,\forall\ h\in\h, k>p.$$
Note that although $a_i\ne a_j$ for $i\ne j$, we can have
$a_i^{n_1}=a_j^{n_1}$ for some $i\ne j$. Let $S_1, S_2, ..., S_s$ be
the partition of $\{1, 2, ..., m\}$ so that
$$i, j\in S_k \ \Leftrightarrow \ a_i^{n_1}= a_j^{n_1}.$$
We deduce  that
$$\sum_{i\in S_j} \l_i(h)a_i^q =0, \,\,\,\forall\ h\in\h,   q=1,2,\cdots, n_1-1.$$
Equivalently, $$\sum_{i\in S_j}\l_i(h)a_i^n=0, \ \forall \ h\in\h, \ n\notin n_1\Z.$$
We see that  $\chi$ is not surjective, contradicting the assumption.  The lemma is true.
\end{proof}

\begin{lem}\label{cal E submodule} 
Let $\pa\in\widetilde{\C^m}$ and
$\pl\in (\h^*)^m$ such that $E(\pl, \pa)$ is finite
dimensional and $E(\pl, \pa, b)$ is irreducible.  Let $W$ be a nonzero $\tg$-submodule of
$\mathcal{E}(\pl,\pa, b, L)$. Then $E(\pl,\pa)\ot L\ot t^n\subseteq
W$ for all sufficiently large $n\in\Z$.
\end{lem}

\begin{proof} Denote $\mathcal{E}=\mathcal{E}(\pl,\pa, b, L)$ and
$\mathcal{E}_n=E(\pl,\pa)\ot L\ot t^n$ for convenience. Elements in
$\mathcal{E}_n$ are called homogeneous. Denote
$f_{\pa}(t)=\prod \limits_{i=1}^m(t-a_i)$ as before. Let $W=\sum \limits_{n\in\Z}W_n\ot t^n$
for some subspaces $W_n\subseteq E(\pl,\pa)\ot L$ for all $n\in\Z$.

 Choose any $n'\in\Z$ such that $W_{n'}\neq 0$ and take a
nonzero $w\in W_{n'}$. We can write $w=\sum \limits_{i=1}^r w_i\ot u_i$ with
$0\neq w_i\in E(\pl,\pa)$ and $u_1,\cdots,u_r\in L$  such that
$u_1,\cdots,u_r$ are linearly independent. There is $p\in\Z_+$ such
that $(\g\ot t^{p}\C[t])u_i=0$ for all $i$.

Regard $E(\pl,\pa)$ as a weight module over $\LL_\pa$ (defined in
(\ref{L_a}))with respect to the Cartan subalgebra
$$\h_\pa=\big(\h\otimes t^p\C[t]\big)/\big((\h\ot t^p\C[t])\cap
(\h\ot f_\pa\C[t])\big).$$ By \eqref{LL_i} and Lemma \ref{iso}, we
see that $$\big(\n_+\otimes t^p\C[t]\big)/\big((\n_+\ot
t^p\C[t])\cap (\n_+\ot f_\pa\C[t])\big)$$ is just the positive part
(nilpotent radical) of $\LL_a$ with respect to $\h_\pa$.

Recall that $E(\pl,\pa)$ is a highest weight module
over the semisimple Lie algebra $\LL_\pa$ with highest weight $\pl$
(see the comments before Remark \ref{remark}). For any weight
$\pmu\in(\h^*)^m$, we say $\pmu\leq\pl$ if $\l_i-\mu_i$ is
nonnegative for all $i=1,\cdots,m$. For any element $v=v_1+\cdots
+v_s\in E(\pl,\pa)$ such that the $v_i$'s are nonzero weight vectors
of distinct weights $\pmu_i$, denote $\mathrm{wt}(v)=\min\{\pmu_i,
i=1,\cdots,s\}$.

If there is $w_i$ so that $\mathrm{wt}(w_i)\neq\pl$, say,
$\mathrm{wt}(w_1)<\pl$, then $w_1$ can not by annihilated by the
whole positive part of $\LL_\pa$. Hence $xw_1\neq 0$ for some
homogeneous $x\in \n_+\otimes t^p\C[t]$. Then we have
$$x(w\ot t^{n'})=x\Big(\sum_{i=1}^r w_i\ot u_i\ot t^{n'}\Big)=\sum_{i=1}^r xw_i\ot u_i\ot t^{n'+\deg(x)}$$
which is nonzero by the linear independence of $u_i, i=1,\cdots,r$.
It is clear that $ \mathrm{wt}(xw_i)
 > \mathrm{wt}(w_i)$ for $ i=1,2,\cdots,r$. Repeating
this process, we can reach a nonzero element $w'=\sum\limits_{i=1}^r
v'_i\ot u_i\ot t^{n''}\in W$ with either $\mathrm{wt}(v'_i)=\pl$ or
$v'_i=0$ for all $i=1,\cdots,r$. Then $w'=v_{\pl}\ot u\ot t^{n''}$
for some nonzero $u\in L$. By replacing $w$ with $w'$
and $n'$ with $n''$, we may assume that 
$w=v_{\pl}\ot u\in W_{n'}$.
\vskip5pt

\noindent{\bf Claim 1.} $E(\pl,\pa)\ot u_1\subseteq  W_n$ for all
sufficiently large $n$. \vskip5pt

From Lemma 7.2,    there exists $k'\in\N$ such that there exist
$h_k\in U_k(\h\otimes (t^p\C[t]))$ with  $\chi(h_k)\neq0$ for all
$k>k'$. By \eqref{LL_a(g)} we can find homogeneous $x_1, x_2, ...,
x_q\in U(\g\ot t^{p}\C[t])\setminus\C$ such that
$E(\pl,\pa)=\spn\{x_iv_{\pl}\ |\ i=1,2,...,q\}$, which is finite
dimensional. Let $\deg(x_i)=l_i$ and $l=\max\{l_i, 1\leq i\leq q\}$.
Then for any $k\geq k'+l$, we have
$$x_ih_{k-l_i}(v_{\pl}\ot u_1\ot t^{n'})=\chi(h_{k-l_i})x_iv_{\pl}\ot u_1\ot t^{n'+k},$$
which implies that $E(\pl,\pa)\ot u_1\ot t^{k}\in W$ for all $k\geq
n'+k'+l$. The claim follows.
\vskip5pt

\noindent{\bf Claim 2.} $E(\pl,\pa)\ot L\subseteq  W_n$ for all
sufficiently large $n$. \vskip5pt

By Claim 1, we suppose that $E(\pl,\pa)\ot u_1\ot t^n\subseteq W$
for all $n\geq n_1$ where $n_1\in\Z$ is fixed. Let $u_1\in
\bigoplus\limits_{i=0}^{m'}L_{-i}$ with minimal $m'\in\Z_+$. Since $L$ is
simple as a $\tg$-module, there exists $y_1,\cdots,y_s\in \g$ and
$j_1,\cdots, j_s\in \Z_+$ with $j_1+\cdots + j_s=m'$  such that
$$w_0= (y_s\ot t^{j_s})\cdots(y_1\ot t^{j_1})u_1\in
L_0\setminus\{0\}.$$ 
By induction on $r$ we see that
$$E(\pl,\pa)\ot (y_r\ot
t^{j_r})\cdots (y_1\ot t^{j_1})u_1\ot t^{n+j_1+...+j_r}\subseteq  W,
\,\, \forall\ n\ge n_1.$$ So
$$E(\pl,\pa)\ot
w_0\ot t^{n+m'}\subseteq  W, \,\,\forall\ n\ge n_1.$$ From Lemma
\ref{L irre} we know that $U(\g\ot\C[t^{-1}])w_0=L$. One can easily
deduce that
$$E(\pl,\pa)\ot
L\ot t^{n}\subseteq  W,  \,\,\forall\ n\ge n_1+m'.$$
Claim 2 follows. Hence the lemma is true.
\end{proof}

The following corollary will be used later in this paper.

\begin{cor}\label{coro} Let $\pa\in\widetilde{\C^m}$ and
$\pl\in (\h^*)^m$ such that $E(\pl, \pa)$ is finite
dimensional and $E(\pl, \pa, b)$ is irreducible.    Let $W$ be a submodule of ${\mathcal E}(\pl,\pa,b,L)$, $n_0\in\Z$ and $u_0\in L_0\setminus\{0\}$.
 If $E(\pl,\pa)\ot u_0\ot t^n\subseteq W$ for all $n\geq n_0$ then $E(\pl,\pa)\ot L\ot t^n\subseteq W$ for all $n\geq n_0$.
\end{cor}
\begin{proof}
Taking $u_1=u_0$ and $n_1=n_0$ in the the proof of Claim 2 of
Theorem \ref{cal E submodule}, we can deduce the result.
\end{proof}

Noticing that $\tn_-=(\g\ot t^{-1}\C[t^{-1}])\oplus\n_-$, we have
$U(\tn_-)=U(\g\ot t^{-1}\C[t^{-1}])\ot U(\n_-).$   Fix
any nonzero $u_0\in L_0$. Then we can define a linear map
\begin{equation}\label{phi}\phi: E(\pl,\pa)\ot U(\tn_-)\rightarrow
E(\pl,\pa)\ot L_0,\end{equation} by $\phi(v\ot xy)=(x\circ v)\ot
yu_0$ for all $v\in E(\pl,\pa)$, $x\in U(\g\ot t^{-1}\C[t^{-1}])$
and $y\in U(\n_-)$, where the action $\circ$ is defined by $1\circ
v=v$ and $$(x_1x_2\cdots x_s)\circ v=(-1)^{s}x_s\cdots x_2x_1v, \
\forall \ x_i\in\g\ot t^{-1}\C[t^{-1}].$$  Note that $L_0$ is an
irreducible $\g$-module. For any $x\in\g\ot t^{-1}\C[t^{-1}], y\in
U(\tn_-)$ and $v\in E(\pl,\pa)$, we can compute that $\phi(xv\ot
y+v\ot xy)=0$.

Both $E(\pl,\pa)\ot U(\tn_-)$ and $E(\pl,\pa)\ot L_0$ are
$U(\n_-)$-modules in a natural way, and we have

\begin{lem}\label{phi is g-homo} The linear map $\phi$ is a $U(\n_-)$-homomorphism.
\end{lem}

\begin{proof} Take any $v\in E(\pl,\pa)$, $x_i\in\g\ot t^{-1}\C[t^{-1}], i=1,\cdots, s, y\in U(\n_-)$ and $x'\in\n_-$. We have
$$\phi(x'(v\ot x_1\cdots x_sy))=\phi(x'v\ot x_1\cdots x_sy)+\phi(v\ot x'x_1\cdots x_sy).$$
Noticing that $\phi(x'v\ot x_1\cdots x_sy)=(-1)^sx_s\cdots x_1x'v\ot yu_0$
and
\begin{equation*}
  \begin{split}
   & \phi(v\ot x'x_1\cdots x_sy)\\
 = & \phi(v\ot x_1\cdots x_sx'y+v\ot \sum_{i=1}^sx_1\cdots[x',x_i]\cdots x_sy)\\
 = & (-1)^sx_s\cdots x_1v\ot x'yu_0+(-1)^s\sum_{i=1}^sx_s\cdots[x',x_i]\cdots x_1 v\ot yu_0\\
 = & (-1)^sx_s\cdots x_1v\ot x'yu_0+(-1)^s[x',x_s\cdots x_1] v\ot yu_0,
   \end{split}
\end{equation*}
we get that
\begin{equation*}
  \begin{split}
   \phi(x'(v\ot x_1\cdots x_sy)) &= (-1)^sx_s\cdots x_1v\ot x'yu_0+(-1)^sx'x_s\cdots x_1v\ot yu_0\\
                                 &= (-1)^sx'(x_s\cdots x_1v\ot yu_0)\\
                                 &= x'\phi(v\ot x_1\cdots x_sy).
   \end{split}
\end{equation*}
as desired.
\end{proof}

\begin{thm}\label{cal E irre} Let $\pa\in\widetilde{\C^m}$ and
$\pl\in (\h^*)^m$ such that $E(\pl, \pa)$ is finite
dimensional and $E(\pl, \pa, b)$ is irreducible. Let $L=\tilde V(\gamma)$ be an irreducible highest weight $\tg$-module with highest weight vector $u_0$ of highest weight $\gamma\in \th^*$. Assume that $J=\Ann_{U(\tn_-)}(u_0)$. Then
$E(\pl,\pa,b)\ot L$ is irreducible if and only if
$\phi(E(\pl,\pa)\ot J)=E(\pl,\pa)\ot L_0$.
\end{thm}

\begin{proof} Consider $\calE(\pl,\pa, b, L)$ instead of $E(\pl,\pa,b)\ot L$.
For convenience denote $E=E(\pl,\pa)$ and $\calE=\calE(\pl,\pa, b,
L)$. Let $W=\sum\limits_{k\in\Z}W_k\ot t^k$ be a nonzero submodule of
$\calE$. By Lemma \ref{cal E submodule}, there is $n\in\Z$ such that
$E\ot L\ot t^i\subseteq W$ for all $i\geq n+1$. \vskip5pt

\noindent{\bf Claim 1.}  For any $v\in E$, $x\in U(\g\ot
t^{-1}\C[t^{-1}])$ and $y\in U(\n_-)$, we have $v\ot xyu_0\equiv
(x\circ v)\ot yu_0=\phi(v\ot xy) \mod W_n$. \vskip5pt

By linearity it suffices to prove the claim for $x=(x_1\ot
t^{-i_1})\cdots (x_s\ot t^{-i_s})$ with $x_j\in \g$ and $i_j\in\N$.
If $s=0$ the result is trivial. Now suppose the result is true for
all $v\in E$, $x=(x_1\ot t^{-i_1})\cdots (x_s\ot t^{-i_s})$ and
$y\in U(\n_-)$. Set $u=yu_0$. Then for any $x_0\in\g$ and $i_0\in\N$
we have that
$$(x_0\ot t^{-i_0})(v\ot xu\ot t^{n+i_0})=\big(v\ot (x_0\ot t^{-i_0})xu
+(x_0\ot t^{-i_0})v\ot xu\big)\ot t^n$$ lies in $W$ and hence
\begin{equation*}\begin{split}
v\ot (x_0\ot t^{-i_0})xyu_0
& \equiv -(x_0\ot t^{-i_0})v\ot xyu_0\mod W_n\\
& \equiv -x\circ (x_0\ot t^{-i_0})v\ot yu_0\mod W_n\\
& \equiv \big((x_0\ot t^{-i_0})x\big)\circ v\ot yu_0\mod W_n.
\end{split}\end{equation*} The claim follows by induction.\vskip 5pt

First suppose that $\phi(E\ot J)=E\ot L_0$. Then for any $v\in
E(\pl,\pa)$,
 we can find $x_i\in J$ and $v_i\in E$ such that
$\phi(\sum\limits_{i=1}^s v_i\ot x_i)=v\ot u_0$. By Claim 1, we have $$v\ot
u_0=\sum_{i=1}^s \phi(v_i\ot x_i)\equiv\sum_{i=1}^s v_i\ot x_i u_0=0
\mod W_n,$$ that is $E\ot u_0\in W_n$. By Corollary \ref{coro}, we obtain that
$W_n=E\ot L$ and by induction on $n$ we conclude that $W_i=E\ot L$
for all $i\in\Z$, that is $W=\calE$. Therefore $\calE$ is
irreducible. \vskip 5pt

Now suppose that $\phi(E\ot J)\ne E\ot L_0$. Then we have an induced
linear map
$$\bar{\phi}: E\ot U(\tn_-)\rightarrow (E\ot L_0)/\phi(E\ot J),\quad v\ot
x\mapsto\overline{\phi(v\ot x)}.$$ Clearly
 $\bar{\phi}(E\ot J)=0$. Noticing that  $L=U(\tn_-)u_0\cong
 U(\tn_-)/J$,
we get the induced linear map $$\widehat{\phi}: E\ot L\rightarrow
(E\ot L_0)/\phi(E\ot J),\quad v\ot xu_0\mapsto\overline{\phi(v\ot
x)}.$$ 

Without loss of generality we may suppose that $W$ is the submodule of $ \calE$ generated by
$E\ot L\ot t^i, i\geq n+1$. Then any elements in $W_n\ot t^n$ is a
sum of elements of the form
$$x(v\ot yu_0\ot t^{n+i}),\quad x\in U(\hg)_{-i}, v\in E, y\in U(\tn_-), i\in\N,$$
and, by the PBW Theorem, is a sum of elements of the form
$$x(v\ot yu_0\ot t^{n+i}),\quad x\in \tn_-, v\in E, \deg(x)=-i, y\in U(\tn_-), i\in\N.$$
While in this case $x(v\ot yu_0\ot t^{n+i})=(xv\ot yu_0+v\ot
xyu_0)\ot t^n$ and $\widehat{\phi}(xv\ot yu_0+v\ot
xyu_0)=\overline{\phi(xv\ot y+ v\ot xy)}=0$. As a result we obtain
that $\widehat{\phi}(W_n)=0$, or equivalently,
$W_n\subseteq\ker\widehat{\phi}\neq E\ot L$ since $\widehat{\phi}$
is surjective and nonzero. Thus $W$ is a proper submodule of
$\calE$, which has to be reducible.
\end{proof}

The above result is not easy  to use since we generally do not know what
$\Ann_{U(\tn_-)}(u_0)$ is for a non-integrable
irreducible highest weight $\tg$-module $\tilde V(\gamma)$ with
highest weight $\gamma\in\th^*$. Fortunately, our purpose is  to handle the case when $\gamma$ is dominant.

From Proposition 9.9 in [K], there are a lot of irreducible highest
weight modules $\tilde V(\gamma)$ that are also Verma modules over
$\tg$. For these irreducible Verma modules $\tilde V(\gamma)$, we
know that 
$J=\Ann_{U(\tn_-)}(u_0)=0$, where $u_0$ is a highest weight vector. Applying Theorem \ref{cal E irre} we obtain

\begin{cor}\label{cal E verma} Let $\pa\in\widetilde{\C^m}$ and
$\pl\in (\h^*)^m$ such that $E(\pl, \pa)$ is finite
dimensional and $E(\pl, \pa, b)$ is irreducible. Let
$\tilde V(\gamma)$ be an irreducible Verma module over $\tg$. Then
$E(\pl,\pa, b)\ot\tilde V(\gamma)$ is always reducible.
\end{cor}

On the other hand, we do have a clear description for
$\Ann_{U(\tn_-)}(u_0)$ if $\gamma$ is dominant. Next we will concentrate on this very
important case which is of most interest to us.

Recall from Section \ref{pre} the notation of root system  relative
to the triangular decomposition of $\g$ and $\tg$ given there. Let
$\a_i, i=1,\cdots,l$ be the simple roots of $\g$, $\theta$ be the
longest root of $\g$ and $\check{\a}_i, \check{\theta}$ be the
corresponding coroots. Let $\a_0=\delta-\theta$. Then $\a_i,
i=0,1,\cdots, l$ are the simple roots of $\tg$. It is well known
that $f_0=e_{\theta}\ot t^{-1}$ and $\check{\a}_0=K-\check{\theta}$.
Any weight $\gamma\in\th^*$ is of the form
$\gamma=\Lambda+k\Lambda_0$, where $\Lambda\in\h^*$ is a weight of
$\g$, $k\in\C$ and $\Lambda_0$ is the root defined by
$\langle\Lambda_0, \check{\a}_i\rangle=\delta_{i,0},
\langle\Lambda_0, d\rangle=0$. Denote by $\tilde V(\gamma)$ the
highest weight $\tg$-module with highest weight $\gamma$ and denote
by $V(\Lambda)$ the  $\g$-submodule of $\tilde V(\gamma)$ generated by the highest weight vector (that is a highest weight $\g$-module  with highest weight
$\Lambda$).

It is well known that $\gamma=\Lambda+k\Lambda_0$ is dominant if and
only if $\Lambda$ is a dominant weight relative to $\g$ and $k\in\N$
satisfies $k\ge \Lambda(\check{\theta})$.

From now on, we assume that $L=\tilde V(\gamma)$ with
$\gamma=\Lambda+k\Lambda_0$ being dominant. It is easy to see that
$L_0=V(\Lambda)$ is just the finite dimensional irreducible highest
weight $\g$-module with highest weight $\Lambda$. Moreover,
$\tilde{V}(\gamma)$ and $V(\Lambda)$ have common highest weight
vector,  and we choose one of them, say $u_{\Lambda}$,
and define the function $\phi$ as in \eqref{phi} replacing $u_0$
with $u_{\Lambda}$. Denote $\mathcal{E}(\pl,\pa, b,
\gamma)=\mathcal{E}(\pl,\pa, b, \tilde{V}(\gamma))$ for convenience.
We have noticed the $\tg$-module isomorphism $\mathcal{E}(\pl,\pa,
b, \gamma)\cong E(\pl,\pa,b)\ot \tilde{V}(\gamma)$.

\begin{thm}\label{cal E irre int} Let $\pa\in\widetilde{\C^m}$ and
$\pl\in (\h^*)^m$ such that $E(\pl, \pa)$ is finite
dimensional and $E(\pl, \pa, b)$ is irreducible. Suppose that $\Lambda+k\Lambda_0\in \th^*$ is dominant. 
Then $E(\pl,\pa,b)\ot \tilde{V}(\Lambda+k\Lambda_0)$ is irreducible
if and only if
\begin{equation}\label{phi()}\phi\Big(E(\pl,\pa)\ot U(\n_-)(e_{\theta}\ot
t^{-1})^{k-\langle\Lambda |
\check{\theta}\rangle+1}\Big)=E(\pl,\pa)\ot
V(\Lambda).\end{equation}
\end{thm}

\begin{proof} Let $u_{\Lambda}$ be a common highest weight vector of $V(\Lambda+k\Lambda_0)$ and $V(\Lambda)$. It is well known that
$J=\Ann_{U(\tn_-)}(u_{\Lambda})$ is the left ideal of $U(\tn_-)$
generated by $f_i^{\langle\Lambda+k\Lambda_0 |
\check{\a}_i\rangle+1}$ for $i=0,1,\cdots,l$. Then
$$\phi(E(\pl,\pa)\ot J)=\phi\Big(E(\pl,\pa)\ot \sum_{i=0}^l U(\g\ot
t^{-1}\C[t^{-1}])U(\n_-)f_i^{\langle\Lambda+k\Lambda_0 |
\check{\a}_i\rangle+1}\Big).$$ Note that
\begin{equation*}\begin{split}
& \phi\Big(E(\pl,\pa)\ot \sum_{i=1}^l U(\g\ot
t^{-1}\C[t^{-1}])U(\n_-)f_i^{\langle\Lambda+k\Lambda_0
| \check{\a}_i\rangle+1}\Big)\\
= & \sum_{i=1}^l U(\g\ot t^{-1}\C[t^{-1}])\circ E(\pl,\pa)\ot U(\n_-)f_i^{\langle\Lambda | \check{\a}_i\rangle+1}u_0= 0.\\
\end{split}\end{equation*}
We have
\begin{equation*}\begin{split}
    \phi(E(\pl,\pa)\ot J)
= & \phi\Big(E(\pl,\pa)\ot U(\g\ot t^{-1}\C[t^{-1}])U(\n_-)f_0^{\langle\Lambda+k\Lambda_0 | \check{\a}_0\rangle+1}\Big)\\
= & \phi\Big(U(\g\ot t^{-1}\C[t^{-1}])\circ E(\pl,\pa)\ot U(\n_-)f_0^
{k-\langle\Lambda| \check{\theta}\rangle+1}\Big)\\
= & \phi\Big(E(\pl,\pa)\ot U(\n_-)f_0^{k-\langle\Lambda| \check{\theta}\rangle+1}\Big).\\
\end{split}\end{equation*}
Now this theorem follows from Theorem \ref{cal E irre}. \end{proof}

The conditions in Theorem \ref{cal E irre int} is still not very
easy to verify. Now we can give a further criterion which we can
compute explicitly.

\begin{thm}\label{cal E irre simpler} Let $\pa\in\widetilde{\C^m}$ and
$\pl\in (\h^*)^m$ such that $E(\pl, \pa)$ is finite
dimensional and $E(\pl, \pa, b)$ is irreducible. Suppose that
$\tilde V(\Lambda+k\Lambda_0)$ is an irreducible integrable highest
weight $\tg$-module, where $\Lambda\in\h^*$  and $k\in\Z_+$. Let $u_{\Lambda}$ be the common
highest weight vector of $\tilde V(\Lambda+k\Lambda_0)$ and $ V(\Lambda)$.
Then $E(\pl,\pa,b)\ot \tilde{V}(\Lambda+k\Lambda_0)$ is irreducible
if and only if \begin{equation}\label{U(n-)}U(\n_-)\big((e_{\theta}\ot
t^{-1})^{k-\langle\Lambda | \check{\theta}\rangle+1}E(\pl,\pa)\ot
u_\Lambda\big)=E(\pl,\pa)\ot V(\Lambda).\end{equation}
\end{thm}

\begin{proof} As before, denote $\gamma=\Lambda+k\Lambda_0$ and we consider $\calE=\calE(\pl,\pa, b, \gamma)$ instead of $E(\pl,\pa,b)\ot \tilde{V}(\Lambda+k\Lambda_0)$.
Denote $E=E(\pl,\pa)$, $f_0=e_{\theta}\ot t^{-1}$ and
$k_0=k-\langle\Lambda | \check{\theta}\rangle+1$ for short.

First suppose that $U(\n_-)(f_0^{k_0}E\ot u_\Lambda)=E\ot
V(\Lambda)$. Let $W$ be a nonzero submodule of $\calE$. As in the
proof of Theorem \ref{cal E irre}, we have $W=\sum\limits_{i\in\Z}W_i\ot
t^i, W_i\subseteq E\ot \tilde V(\gamma)$, and by Lemma \ref{cal E
submodule}, there is $n\in\Z$ such that $W_i=E\ot \tilde V(\gamma)$
for all $i\geq n+1$.

By Claim 1 of Theorem \ref{cal E irre}, we have that, for any $v\in
E$,
$$(-1)^{k_0}f_0^{k_0}v\ot u_\Lambda=\phi(v\ot f_0^{k_0})\equiv v\ot f_0^{k_0}u_{\Lambda}=0\mod W_n.$$
So $f_0^{k_0}E\ot u_\Lambda\subseteq W_n$ and hence $E\ot
V(\Lambda)=U(\n_-)(f_0^{k_0}E\ot u_\Lambda)\subseteq W_n$, which
implies $E\ot \tilde V(\gamma)=W_n$ by Corollary \ref{coro}. By
induction, we can obtain $W_i=E\ot \tilde V(\gamma)$ for all
$i\in\Z$. Therefore $\calE$ is irreducible.

Next we suppose that $\calE$ is irreducible. To complete the proof
it suffices to show the following claim, thanks to Theorem \ref{cal
E irre int} and the fact that $
U(\n_-)(f_0^{k_0}E\ot u_{\Lambda})\subseteq E(\pl,\pa)\ot V(\Lambda)$.\vskip5pt

\noindent{\bf Claim 1.} $\phi(E\ot U(\n_-)f_0^{k_0})\subseteq
U(\n_-)(f_0^{k_0}E\ot u_{\Lambda})$.\vskip5pt

Take any $x=x_1x_2\cdots x_r\in U(\n_-)$ with $x_i\in \n_-$. We will
prove
\begin{equation}\label{last}
\phi(E\ot x_1\cdots x_rf_0^{k_0})\subseteq U(\n_-)(f_0^{k_0}E\ot
u_{\Lambda})
\end{equation}
by induction on $r$. The assertion is obvious for $r=0$. Now suppose
that \eqref{last} holds for $r\in\Z_+$. Then for any $v\in E$ and
$x_0\in \n_-$, we have
\begin{equation*}\begin{split}
    \phi(v\ot x_0xf_0^{k_0}) & =\phi(x_0(v\ot xf_0^{k_0})-x_0v\ot xf_0^{k_0})\\
                             & =x_0\phi(v\ot xf_0^{k_0})-\phi(x_0v\ot xf_0^{k_0}),
\end{split}\end{equation*} since $\phi$ is a $U(\n_-)$-module homomorphism. Then
by induction hypothesis, we have $\phi(v\ot xf_0^{k_0}),
\phi(x_0v\ot xf_0^{k_0})\in U(\n_-)(f_0^{k_0}E\ot u_{\Lambda})$.
Hence $\phi(v\ot x_0xf_0^{k_0})\in U(\n_-)(f_0^{k_0}E\ot
u_{\Lambda})$. The claim follows.
\end{proof}

It is interesting to note that the left hand side of
\eqref{U(n-)} is a $\g$-module (even if the equality in \eqref{U(n-)} does not
hold), while the vector space inside $\phi(\cdot)$ on the left hand
side of \eqref{phi()} is not a $\g$-module in general.

From Theorem \ref{cal E irre simpler}, the condition \eqref{U(n-)} is
computable since everything is within the finite dimensional
$\g$-module $$E(\pl,\pa)\ot V(\Lambda)=V(\lambda_1)\ot V(\lambda_2)\ot\cdots\ot V(\lambda_m)\ot V(\Lambda).$$
Now the
main task is to compute the highest weight vectors with respect to
$\g$ in the submodule $U(\n_-)\left((e_{\theta}\ot t^{-1})^{k-\langle\Lambda |
\check{\theta}\rangle+1}E(\pl,\pa)\ot u_\Lambda\right)$.

As an application of  the above theorem, we can obtain the
irreducibility for $\tilde{\sl_2}$ explicitly, recovering the result
of Adamovic \cite{Ad2}. Choose a standard basis $\{e, f, h\}$ of
$\sl_2$ and let $\eps$ be the fundamental weight.

\begin{cor}\label{coro sl_2}
For any $a, b\in\C$ and $i,j, k\in\Z_+$ with $a\neq0$ and $k\geq i$, the
$\tilde{\sl}_2$-module $E(j\eps,a,b)\ot \tilde V(i\eps+k\Lambda_0)$ is
irreducible if and only if $k<j$.
\end{cor}

\begin{proof} We know that $\g=\sl_2=\spn\{e, f, h\}$,   $e_0=f\ot t$ and $f_0=e\ot t^{-1}$.
Denote  by $v$ and $u$ the highest weight
vectors of the $\g$-module $V(j\eps)$ and the $\tg$-module
$\tilde V(i\eps+k\Lambda_0)$ respectively. Denote $W=V(j\eps)\ot V(i\eps)$
and $W'=U(\n_-)\big(f_0^{k-i+1}E(j\eps, a)\ot u\big)$ for
convenience. By Theorem \ref{cal E irre simpler}, we see that
$E(j\eps,a,b)\ot \tilde V(i\eps+k\Lambda_0)$ is irreducible if and only if
$W'=W$ if and only if $\C[f]\left( f_0^{k-i+1}E(j\eps, a)\ot u\right)=\C[f]\left(e^{k-i+1}V(j\eps)\ot u\right)$ contains all
highest weight vectors of $V(j\eps)\ot V(i\eps)$.

First note that $e^{k-i+1}E(j\eps,a)=0$ and $W'\neq W$ if $j<k-i+1$.

Next we suppose $j\geq k-i+1$. In this case, we have
$$f_0^{k-i+1}E(j\eps, a)=\spn\{ v, fv, \cdots, f^{j-(k-i+1)}v\}.$$
Let $W'_p$ (resp. $W_p$) be the weight space of $W'$ (resp. $W$) of
weight $p\epsilon$. For $s=0,1,\cdots, \min\{i,j\}$, we can deduce
that
$$\dim W'_{i+j-2s}=
\left\{\begin{array}{ll}
s+1,   & \text{if}\ s\leq j-(k-i+1),\\\\
j-k+i, & \text{if}\ s>j-(k-i+1).
\end{array}\right.$$
On the other hand, by the Clebsch-Gordan's formula, we have
$$W=V((i+j)\eps)\oplus V((i+j-2)\eps)\oplus\cdots\oplus
V(|j-i|\eps).$$ It is clear that $W_{i+j}\oplus\cdots\oplus
W_{|i-j|}$ contains all highest weight vector of $W$ and $\dim
W_{i+j-2s}=s+1$ for all $s=0,1,\cdots,\min\{i,j\}$.

Then $W'=W$ if and only if $W'_{i+j-2s}=W_{i+j-2s}$ for all $0\leq
s\leq \min\{i,j\}$, if and only if $j-(k-i+1)\geq\min\{i,j\}$, if
and only if $k+1\leq\max\{i,j\}$. Noticing that  $\tilde
V(i\eps+k\Lambda_0)$ being dominant requires $k\geq i$, we can
conclude that $E(j\eps,a,b)\ot \tilde V(i\eps+k\Lambda_0)$ is
irreducible if and only if $k<j$.
\end{proof}

Let us continue the notations as in Theorem \ref{cal E irre
simpler}. In particular, $\gamma=\Lambda+k\Lambda_0$. Denote by
$E_{\pmu}$ the weight space of the $\g$-module $E(\pl,\pa)\otimes
V(\Lambda)$ with weight $\pmu\in(h^*)^{m+1}$. For any root $\a$ of
$\g$, let $E^{(\a)}$ be the sum of all $E_{\pmu}$ with $
\pmu=(\mu_1,\cdots,\mu_{m+1})$ satisfying
$k+\brac{\mu_1+\cdots+\mu_{m+1}}{\check{\a}}<0.$

\begin{prop}\label{chari} Let notations be as in Theorem \ref{cal E irre simpler}. Let $W=\sum\limits_{i\in\Z}W_i\ot t^i$ be a submodule of $\calE=\calE(\pl,\pa, b, \Lambda+k_0\Lambda_0)$ such that each $W_i$ is a subspace of
$E(\pl,\pa)\otimes \tilde{V}(\Lambda+k_0\Lambda)$ and $W_i=E(\pl,\pa)\ot \tilde{V}(\Lambda+k_0\Lambda)$ for all $i> n$. Then $E^{(\a)}\subseteq W_n$ for any  root $\a$ of $\g$.
\end{prop}

\begin{proof}
Suppose $E_{\pmu}\not\subseteq W_n$ for some weight $\pmu=(\mu_1,\cdots,\mu_{m+1})\in(\h^*)^{m+1}$.
Take any $v\in E_{\pmu}\setminus W_n$.

  For any   root  $\a$ of $\g$, take nonzero $x\in \g_\a$. We have
$$(x\ot t)(v\ot t^n)\in W_{n+1}.$$
That is, the image of $v\ot t^n$ in $\calE/W$ is a highest weight vector of the integrable module over the Lie algebra $\sl_2$ spanned by $x\ot t, y\ot t^{-1}$ and $\check{\a}+K$, where $y$ is a nonzero root vector in $g_{-\a}$.
So we have $\brac{\sum\limits_{i=1}^{m+1}\mu_i+k\Lambda_0+(b+n)\delta}{\check{\a}+K}\geq0$, or equivalently
$$k+\brac{\mu_1+\mu_2+\cdots+\mu_{m+1}}{\check{\a}}\geq0, \forall \alpha\in\Delta.$$
The result hence follows.
\end{proof}

By the above proposition, we see that $$U(\g)\Big(\sum_{\a\in\Delta}E^{(\a)}\big)\subseteq W_n.$$
Using this and Claim 1 in the proof of Theorem 7.6, we obtain that  $\calE$ is irreducible if and only if
$$U(\g)\Big(\sum\limits_{\a\in\Delta}E^{(\a)}+(e_{\theta}\ot t^{-1})^{k+1-\brac{\Lambda}{\check{\theta}}}E(\pl,\pa)\ot u_\Lambda\Big)=E(\pl,\pa)\ot \tilde{V}(\Lambda+k\Lambda_0).$$

The following consequence is easy to see:

\begin{cor}\label{coro2}  Under the conditions in Theorem \ref{cal E irre simpler}, the following hold.
\begin{enumerate}
\item If $k\ge\brac{\lambda_1+\lambda_2+...+\lambda_m+\Lambda}{\check{\theta}}$, then $E(\pl,\pa,b)\ot \tilde{V}(\Lambda+k\Lambda_0)$
is not irreducible.
\item If $k<\brac{\mu_1+\cdots+\mu_{m+1}}{\check{\theta}}$ for each highest weight $\pmu=(\mu_1,\cdots,\mu_{m+1})$  of the $\g$-module $E(\pl,\pa)\otimes V(\Lambda)$, then $\tg$-module
$E(\pl,\pa,b)\ot \tilde{V}(\Lambda+k\Lambda_0)$
is irreducible.
\end{enumerate}
\end{cor}

\begin{proof} If $k\ge\brac{\lambda_1+\lambda_2+...+\lambda_m+\Lambda}
{\check{\theta}}$, the left hand side of \eqref{last} is $0$. So $E(\pl,\pa,b)\ot \tilde{V}(\Lambda+k\Lambda_0)$
is not  irreducible. Part (1) follows.

For Part (2), we see from Proposition \ref{chari} that $E_{\pmu}\subseteq E^{(-\theta)}$ and the result follows from the remark below
Proposition \ref{chari}.
\end{proof}

\begin{cor}\label{coro_k_1} Let notations be as in Theorem \ref{cal E irre simpler}, and $\pl,\pa, b, \Lambda$ be given.
Then there exists $k_1\in\Z_+$, depending on  $\pl,\pa, b, \Lambda$,  such that the $\tg$-module $E(\pl,\pa, b)\ot V(\Lambda+k\Lambda_0)$, where $k\in\Z_+$ with $k\ge \langle\Lambda, {\check{\theta}}\rangle$,  is irreducible if and only if  $k< k_1$.
\end{cor}

\begin{proof} By Theorem \ref{cal E irre simpler}, it is sufficient to sow that
$$(e_{\theta}\ot t^{-1})^jE(\pl,\pa)\subseteq (e_{\theta}\ot t^{-1})^{j'}E(\pl,\pa),\ \forall\ j>j'.$$
This can be seen from the fact  that
$$(e_{\theta}\ot t^{-1})^{j-j'}E(\pl,\pa)\subseteq E(\pl,\pa),\ \forall\ j>j'.\qedhere$$
\end{proof}

Note that the integer $k_1$ in the above corollary depends only on
$\pl,\pa$ and $\Lambda$. We denote it by $\kappa(\pl,\pa,\Lambda)$.
From Corollary \ref{coro2}, we see that
$${\brac{\mu_1+\cdots+\mu_{m+1}}{\check{\theta}}}\leq \kappa(\pl,\pa,\Lambda)<
\brac{\l_1+\cdots+\l_m+\Lambda}{\check{\theta}},$$
for each highest weight $\pmu=(\mu_1,\cdots,\mu_{m+1})$  of the tensor product $\g$-module $E(\pl,\pa)\otimes V(\Lambda)$.

From Corollary \ref{coro sl_2}, we see that $\kappa(j\epsilon, a,
i\epsilon)=j$ for any $a\in\C^*$ and $i,j\in\Z_+$ in case of $m=1$
for the algebra $\tilde{\sl}_2$. Now we look at another example for
$\tilde{\sl}_2$.\smallskip

\noindent\textbf{Example.}  Let $m=2, \pl=(\epsilon,2\epsilon),
b\in\C$ and $\bold{a}=(a_1, a_2)$ with distinct $a_1, a_2\in\C^*$.
We have the $\tilde{\sl}_2$-module $E((\epsilon,2\epsilon); (a_1,
a_2); b)$ and the highest weight $\tilde{\sl}_2$-module $\tilde
V(i\epsilon+k\Lambda_0)$ where nonnegative integers $k\ge i$. Now we
consider the the $\tilde{\sl}_2$-modules $E((\epsilon,2\epsilon);
(a_1, a_2); b)\ot \tilde V(i\eps+k\Lambda_0)$.

Note that the function $\chi$ defined in \eqref{chi} is always onto.

Let $v_1, v_2$ and $u$ be the highest weight vectors of the
$\sl_2$-modules $V(\eps), V(2\eps)$ (in the tensor product of
$E((\epsilon,2\epsilon); (a_1, a_2))$) and the
$\tilde{\sl}_2$-module $\tilde V(i\eps+k\Lambda_0)$ respectively.
Then $u$ is also the highest weight vector of the $\sl_2$-submodule
$V(i\eps)$ in the $\tilde{\sl}_2$-module $\tilde
V(i\epsilon+k\Lambda_0)$.
For convenience, denote $$W=E((\eps, 2\eps); (a_1, a_2))\ot V(i\eps)\cong\big(V(3\eps)\oplus V(\eps)\big)\ot V(i\eps)$$ and $$W^{(j)}=\C[f]\big((e\ot t^{-1})^j\big(V(\eps)\ot
V(2\eps)\big)\ot u\big),\ \forall\ j\in\Z_+.$$
Then the $\tilde\sl_2$-module
$E((\epsilon,2\epsilon); (a_1, a_2); b)\ot
\tilde V(i\eps+k\Lambda_0)$ is irreducible if and only if $W^{(k-i+1)}=W$.
To compute $W^{(j)}$, we need the following formulas
$$\aligned &efv_1=v_1, \ efv_2=2v_2,  \ ef^2v_2=2fv_2,\\  &f^2v_1=f^3v_2=0, \ efu=iu.\endaligned$$
It is not hard to calculate that
$$\aligned
  & (e\ot t^{-1})\big(V(\eps)\ot V(2\eps)\big)\\
        &\hskip 10pt    = \text{span}\{v_1\ot v_2, \   v_1\ot fv_2, \   fv_1\ot v_2,\ a_2v_1\ot f^2v_2 +2a_1fv_1\ot fv_2\};\\
   & (e\ot t^{-1})^2\big(V(\eps)\ot V(2\eps)\big)=\text{span}\{v_1\ot v_2, \  a_2 v_1\ot fv_2+a_1 fv_1\ot v_2\};\\
   & (e\ot t^{-1})^3\big(V(\eps)\ot V(2\eps)\big)=\C v_1\ot v_2.
 \endaligned$$

\vskip5pt\noindent\textbf{Case 1}: $i=0$.\vskip5pt

In this case, it is obvious that $W^{(2)}=W$ and $W^{(3)}\neq W$
and hence $E((\eps, 2\eps); (a_1, a_2), b)\ot \tilde V(\Lambda_0)$ is an irreducible $\tilde{\sl_2}$-module,
while
 $E((\eps, 2\eps); (a_1, a_2), b)\ot\tilde V(2\Lambda_0)$ is not irreducible.
We conclude that $\kappa((\eps, 2\eps); (a_1, a_2); 0)=2$ for all
distinct $a_1, a_2\in\C^*$.

\vskip5pt \noindent\textbf{Case 2}: $i=1$.\vskip5pt

The vector space $\C f^2(v_1\ot v_2\ot u)+\C f(v_1\ot fv_2\ot u)+\C
f(fv_1\ot v_2\ot u)+\C(a_2v_1\ot f^2v_2\ot u +2a_1fv_1\ot fv_2\ot
u)$ contains the following linearly independent weight-$0$ vectors
$$\aligned
&v_1\ot f^2v_2\ot u, \ fv_1\ot fv_2\ot u, fv_1\ot v_2\ot fu, \
v_1\ot fv_2\ot fu.
\endaligned$$
We obtain $4\leq \dim W^{(1)}_0\leq \dim W_0=4$. Hence
$W_0=W^{(1)}_0$ and $W^{(1)}=W$. The $\tilde{\sl}_2$-module
$E((\epsilon,2\epsilon); (a_1, a_2); b)\ot \tilde V(\eps+\Lambda_0)$
is irreducible. On the other hand, we can easily see that $\dim
W^{(2)}_2\leq 2$ while $\dim W_2=3$. So $W^{(2)}\neq W$ and the
$\tilde{\sl}_2$-module $E((\epsilon,2\epsilon); (a_1, a_2); b)\ot
\tilde V(\eps+2\Lambda_0)$ is not irreducible. We obtain again that
$\kappa((\eps, 2\eps); (a_1, a_2); \eps)=2$ for all distinct $a_1,
a_2\in\C^*$.

\vskip5pt \noindent\textbf{Case 3}: $i\geq 2$.\vskip5pt

In this case, we see similarly that $\dim W^{(1)}_{i-1}= 4$ and
$\dim W_{i-1}=5$. So $W^{(1)}\neq W$ and $E((\epsilon,2\epsilon);
(a_1, a_2); b)\ot \tilde V(i\eps+k\Lambda_0)$ is not irreducible for
$k=i$ and hence for all $k\geq i$. We obtain again that
$\kappa((\eps, 2\eps); (a_1, a_2); i\eps)=i$ for all distinct $a_1,
a_2\in\C^*$. \qed\vskip5pt

Note that the irreducible modules for $i=0, k=1$ in the above example does not satisfy the conditions in Theorem 2.2 in [CP3]:
$$\aligned
&a_1\lambda_1+a_2\lambda_2+...+a_m\lambda_m\ne0, \text{ or }   \\
&(m+1)k<\langle \lambda_1+\lambda_2+...+\lambda_m+\Lambda,
\check{\theta}\rangle.\endaligned$$ It is also interesting that in
the above examples, the values of the function $\kappa$ is
independent of $a_1$ and $a_2$.

\

We will conclude the paper with the following problem:
\smallskip

\noindent\textbf{Problem}. Is there an explicit formula for the
function $\kappa(\pl,\pa,\Lambda)$ defined in Corollary 7.13 in
terms of $\pl,\pa$ and $\Lambda$.

\newpage

\begin{center}
\bf Acknowledgments
\end{center}

\noindent The first part presented in this paper was carried out
during the visit of X.G. to Wilfrid Laurier University, while the
second part was carried out during the visit of K.Z. to
Mittag-Leffler Institute of Mathematics, Sweden. X.G. is partially
supported by NSF of China (Grant 11471294) and the Foundation for
Young Teachers of Zhengzhou University (Grant 1421315071). K.Z. is
partially supported by NSF of China (Grant 11271109) and NSERC
(Grant 311907-2015). The authors like to express their thanks to
Prof. V. Mazorchuk for suggestions to the old version of this paper.

\vspace{0.5cm} \noindent X.G.: Department of Mathematics, Zhengzhou
University, Zhengzhou 450001, Henan, P. R. China. Email:
guoxq@amss.ac.cn

\vspace{0.1cm} \noindent K.Z.: Department of Mathematics, Wilfrid
Laurier University, Waterloo, ON, Canada N2L 3C5,  and College of
Mathematics and Information Science, Hebei Normal (Teachers)
University, Shijiazhuang, Hebei, 050016 P. R. China. Email:
kzhao@wlu.ca


\begin{thebibliography}{99}

\bibitem [Ad1] {Ad1} D. Adamovic, New irreducible modules for affine Lie algebras at the critical level, International Math. Research Notices, (1996), 253-262.

\bibitem [Ad2] {Ad2} D. Adamovic, Vertex operator algebras and irreducibility of certain mod-
ules for affine Lie algebras, Math. Research Letters, 4 (1997), 809-821.

\bibitem [Ad3] {Ad3} D. Adamovic, An application of U(g)-bimodules to representation theory
of affine Lie algebras, Algebras and Rep. theory, 7 (2004), 457-469.

\bibitem [ALZ] {ALZ}  D. Adamovic, R. Lu, K. Zhao, Whittaker modules for the affine Lie algebra $A_1 ^{(1)}$,  Adv. Math., 289(2016), 438-479.


\bibitem[BM]{BM} P. Batra, V. Mazorchuk, Blocks and modules for Whittaker pairs, J. Pure Appl. Alg., {\bf 215} (2011) 1552--1568.

\bibitem[BBFK]{BBFK} V. Bekkert, G. Benkart, V. Futorny, I. Kashuba,  New irreducible modules for Heisenberg and affine Lie algebras, J. Algebra, 373 (2013), 284-298.

\bibitem[B]{B} R. Block, The irreducible representations of the Lie algebra $\sl_2$ and
of the Weyl algebra, Adv. Math., 39 (1) (1981) 69--110.

\bibitem[Ca]{Ca} R. W. Carter, Lie algebras of finite and affine type, Cambridge Studies in Advanced Mathematics, no. 96, Cambridge University Press, Cambridge, 2005.

 \bibitem[Ch]{Ch}   V. Chari, Integrable representations of affine Lie-algebras, Invent. Math., 85 (1986), 317-335.

\bibitem[CP1]{CP1} V. Chari and A. Pressley, New unitary representations of loop groups, Math. Ann., 275 (1986), 87-104.

\bibitem[CP2]{CP2}  V. Chari and A. Pressley, Integrable representations of twisted affine Lie algebras, J. Algebra 113
(1988), 438-464.

\bibitem[CP3]{CP3} V. Chari and A. Pressley, A new family of irreducible, integrable modules for affine Lie
algebras, Math. Ann.,    277  (1987),  no. 3, 543-562.

\bibitem[CP4]{CP4} V. Chari and A. Pressley, An application of Lie superalgebras to affine Lie algebras, J. Algebra,  135  (1990),  no. 1, 203-216.

\bibitem[CGZ]{CGZ}
H. Chen, X. Guo, K. Zhao, Tensor product weight modules over the
Virasoro Algebra,   J. London Math. Soc., (2013), 88 (3):
829-844.

\bibitem[Chr]{Chr} K. Christodoulopoulou, Whittaker modules for Heisenberg algebras and imaginary Whittaker modules for affine Lie algebras, J. Algebra, {\bf 320} (2008) 2871-2890.

\bibitem[DFG]{DFG} I. Dimitrov, V. Futorny, D. Grantcharov, Parabolic sets of roots, in: Groups, Rings and Group Rings, in: Contemp. Math.,
vol. 499, Amer. Math. Soc., Providence, RI, 2009, pp. 61-73.

\bibitem[DG]{DG} Ivan Dimitrov, Dimitar Grantcharov, Classification of simple weight modules over affine Lie algebras, arXiv:0910.0688.

\bibitem[D]{D} J.~Dixmier, Enveloping algebras. Revised reprint of the 1977 translation. Graduate Studies
in Mathematics, {\bf 11}. American Mathematical Society, Providence,
RI, 1996.

\bibitem[E1]{E1} S. Eswara Rao, On representations of loop algebras, Commun. Algebra, 21(1993), 2131-2153.

\bibitem[E2]{E2} S. Eswara Rao, Classification of loop modules with finite-dimensional weight spaces,
Math. Ann., 305 (1996), 651-663.

\bibitem[F1]{F1}
V. Futorny, The parabolic subsets of root systems and corresponding
representations of affine Lie algebras, in: Proceedings of the
International Conference on Algebra, Part 2, Novosibirsk, 1989, in:
Contemp. Math., vol. 131, Amer. Math. Soc., Providence, RI, 1992,
pp. 45-52.

\bibitem[F2]{F2} V. Futorny, Imaginary Verma modules for affine Lie algebras, Canad. Math. Bull., 37 (1994) 213Ã¯Â¿Âœ218.

\bibitem[F3]{F3}V. Futorny, Irreducible non-dense $A^{(1)}_1$-modules, Pacific J. Math., 172 (1996) 83-99.

\bibitem[F4]{F4} V. Futorny, Representations of Affine Lie Algebras, QueenÃ¯Â¿Âœs Papers in Pure and Appl. Math., vol. 106, QueenÃ¯Â¿Âœs University,
Kingston, ON, 1997.

\bibitem[F5]{F5} V. Futorny, Representations of affine Lie superalgebras, Groups, rings and group rings, 163-172,
Lect. Notes Pure Appl. Math. 248.

\bibitem[FGM]{FGM} V. Futorny,  D. Grantcharov,  R. A. Martins,  Localization of free field realizations of affine Lie algebras. Lett. Math. Phys.,  105  (2015),  no. 4, 483-502.


\bibitem[FT]{FT} V. Futorny and A. Tsylke, Classification of irreducible nonzero level modules with finite--dimensional weight spaces for affine Lie algebras, J. Algebra 238 (2001) 426-441.

\bibitem[GB]{GB}  L. C. Grove, C. T. Benson, Finite reflection groups, Second edition. Graduate Texts in Mathematics, no. 99, Springer-Verlag, New York, 1985.

\bibitem[Hu]{Hu} J. E. Humphreys, Introduction to Lie Algebras and Representation Theory, Graduate Texts in Mathematics, no. 9, Springer-Verlag, New York-Berlin, 1978.

\bibitem[Ja]{Ja} N. Jacobson,   Basic algebra. II. W. H. Freeman and Co., San Francisco, Calif., 1980.

\bibitem[JK]{JK} H. P. Jakobsen, V. Kac, A new class of unitarizable highest weight representations of infinite dimensional Lie algebras, in: Nonlinear Equations in Classical and Quantum Field Theory, Meudon/Paris, 1983/1984, in: Lecture Notes in Phys.,
vol. 226, Springer, Berlin, 1985, pp. 1Ã¯Â¿Âœ20.

\bibitem[K]{K}  V. Kac, {Infinite dimensional Lie algebras, 3rd edition},
Cambridge Univ. Press, 1990.

\bibitem[Ko]{Ko} B. Kostant, On Whittaker vectors and representation theory, Invent. Math., {\bf 48} (1978) 101-184.

\bibitem[L]{L} M. Lau, Representations of multiloop algebras, Pacific J. Math.,
245 (2010), no. 1, 167-184.

\bibitem[Li]{Li} H. Li, On certain categories of modules for affine Lie algebras, Math. Z., {\bf 248} (2004), no.3, 635-664.

\bibitem[MZ1]{MZ} V. Mazorchuk, K. Zhao, Characterization of simple highest weight modules,  arXiv:1105.1123.
Canadian Mathematical Bulletin, {\bf 56} (2013), n.3,   606-614.

\bibitem[MZ2]{MZ2} V.~Mazorchuk, K.~Zhao, Simple Virasoro modules which are locally finite over a positive part,     {Selecta Mathematica},  20(2014),  no.3, 839-854.

\bibitem[Mc]{Mc} E. McDowell, On modules induced from Whittaker modules, J. Algebra, {\bf 96} (1985) 161-177.

\bibitem[MP]{MP} R.~Moody, A.~Pianzola, Lie algebras with triangular decompositions. Canadian Mathematical Society Series of
Monographs and Advanced Texts. A Wiley-Interscience Publication.
John Wiley \& Sons, Inc., New York, 1995.




\end{thebibliography}
\end{document}